\documentclass[a4paper,11pt,reqno]{amsart}
\usepackage{enumerate}
\usepackage{amssymb, amsmath}
\usepackage{mathrsfs}
\usepackage{amscd}
\usepackage{xfrac}
\usepackage[active]{srcltx}
\usepackage{verbatim}
\usepackage[colorlinks,linkcolor={black},citecolor={blue},urlcolor={black}]{hyperref}

\theoremstyle{plain}
\newtheorem{theorem}{Theorem}[section]

\newtheorem{lemma}[theorem]{Lemma}
\newtheorem{proposition}[theorem]{Proposition}

\newtheorem{assumption}[theorem]{Assumption}

\newtheorem*{definition*}{Definition}
\newtheorem*{theorem*}{Theorem}

\theoremstyle{remark}
\newtheorem{remark}[theorem]{Remark}

\newtheorem*{claim*}{Claim}
\newtheorem*{remark*}{Remark}
\newtheorem*{example*}{Example}
\newtheorem*{notation*}{Notation}

\numberwithin{equation}{section}

\def\R{{\mathbb R}}
\def\C{{\mathbb C}}
\def\H{{\mathbb H}}

\newcommand{\eps}{\varepsilon}
\renewcommand{\phi}{\varphi}

\newcommand{\dd}{\; \mathrm{d}}

\newcommand{\bra}[1]{[ {#1}]}

\newcommand{\norm}[1]{\| {#1}\|}

\DeclareMathOperator{\diverg}{div}

\DeclareMathOperator{\Ric}{Ric}

\DeclareMathOperator{\Leb}{Leb}
\DeclareMathOperator{\CC}{cc}
\newcommand{\dcc}{d_{\mathrm{cc}}}

\newcommand{\ddt}{\frac{\mathrm{d}}{\mathrm{d}t}}

\newcommand{\vx}{{{\mathbf X}}}
\newcommand{\vy}{{{\mathbf Y}}}
\newcommand{\vu}{{{\mathbf U}}}

\newcommand{\vv}{{{\mathbf V}}}

\newcommand{\cL}{\mathcal{L}}

\newcommand{\cP}{\mathscr{P}}

\renewcommand{\tilde}{\widetilde}

\DeclareMathOperator{\Vect}{Vect}

\newcommand{\Rm}{\mathrm{Riem}}

\DeclareMathOperator{\Ent}{Ent}
\DeclareMathOperator{\AC}{AC}
\DeclareMathOperator{\vol}{vol}
\DeclareMathOperator{\proj}{pr}

\newcommand{\h}{\mathfrak{h}}

\begin{document}

\title[]{Smoothing and non-smoothing via a flow tangent to the Ricci flow}

\author{Matthias Erbar} \author{Nicolas Juillet} \address{
  University of Bonn\\
  Institute for Applied Mathematics\\
  Endenicher Allee 60\\
  53115 Bonn\\
  Germany} 
\email{erbar@iam.uni-bonn.de}
\address{Institut de Recherche Math\'ematique Avanc\'ee\\
  UMR 7501\\
 Universit\'e de Strasbourg et CNRS\\
 7 rue Ren\'e  Descartes\\
 67\,000 Strasbourg\\
 France}
\email{nicolas.juillet@math.unistra.fr}



 \begin{abstract}
   We study a transformation of metric measure spaces introduced by
   Gigli and Mantegazza consisting in replacing the original
   distance with the length distance induced by the transport
   distance between heat kernel measures. We study the smoothing
   effect of this procedure in two important examples. Firstly, we
   show that in the case of some Euclidean cones, a singularity persists at the apex. Secondly, we generalize the construction to a
   sub-Riemannian manifold, namely the Heisenberg group, and show that
   it regularizes the space instantaneously to a smooth Riemannian
   manifold.
 \end{abstract}


\date{}

\maketitle
 
\section{Introduction}
\label{sec:intro}

There are many ways to deform a Riemannian manifold into a singular
metric space as discussed for instance in the influential essay of
Gromov \cite{Gro}.  We are interested in the opposite question whether
there exists a deformation, intrinsically defined for a wide class of
metric spaces that instantaneously turns the space into a Riemannian
manifold. In this paper, we investigate a method that has been
introduced by Gigli and Mantegazza \cite{GM}. We examine its
regularization properties in two important cases: Euclidean cones and
the Heisenberg group. These are emblematic examples of Alexandrov
spaces and subRiemannian spaces respectively. We also discuss normed
vector spaces where the transformation turns out to be the identity as
an example of Finsler structures.

Before we state our results we briefly explain the main features of
the construction of Gigli and Mantegazza which is based on the
interplay of optimal transport and Ricci curvature. The starting point
is a metric measure space $(X,d,m)$ on which a reasonable notion of
heat kernel can be defined. For $t>0$ a new distance $d_t(x,y)$ is defined
as the length distance induced by the $L^2$ Wasserstein distance
built from $d$ between the heat kernel measures centered at $x$ and
$y$. 

The striking feature of this approach is the following main result of
\cite{GM}: When $(X,d,m)$ is a Riemannian manifold then $d_t$ is induced
by a smooth metric tensor $g_t$ that is tangent to the Ricci flow,
i.e.~$\partial_t|_{t=0}g_t=-2\Ric$ in a weak sense. Gigli and
Mantegazza then generalize this construction to metric measures spaces
with generalized Ricci curvature lower bounds, namely the RCD
condition, which ensures existence of a well-behaved heat kernel.
This can be seen as a first step into constructing a Ricci flow for non-smooth initial data. A related synthetic characterization of super-Ricci flows based on optimal transport has been obtained by McCann and Topping \cite{McCT}.

One can think of $d_t$ as a sort of convolution of the original
distance with the heat kernel. Having the smoothing effect of the heat
equation and Ricci flow in mind, one might expect that this procedure
gives a canonical way of regularizing the metric measure space.

A first study of the regularizing effects of the Gigli-Mantegazza flow
has been performed by Bandara, Lakzian and Munn \cite{BLM} in the case
where the distance $d$ is induced by a metric tensor with low
regularity and isolated conic singularities. It is shown that $d_t$ is
induced by a metric tensor with at least the same regularity away from
the original singular set. The question, what happens at the
singularities has been left unanswered.

In the present paper, we give an answer showing that conic
singularities can persist under the Gigli-Mantegazza
transformation. We analyse in detail the transformation for two
specific Euclidean cones of angle $\pi$ and $\pi/2$. Our results are
the following (see Theorem \ref{thm:singular} and Proposition \ref{prop:convergence1} below).

\begin{theorem}\label{thm:main-cone1}
  Let $C(\pi)$ be the two-dimensional Euclidean cone of angle $\pi$
  and $d$ its distance. For every $t>0$ the convoluted distance $d_t$
  has a conic singularity of angle $\sqrt{2}\pi$ at the apex.

  As $t$ goes to zero, the metric space $(C(\pi),d_t)$ tends to
  $(C(\pi),d)$ pointwise and in the pointed Gromov--Hausdorff
  topology. As $t$ goes to infinity, it tends to the Euclidean cone of angle $\sqrt{2}\pi$ in the pointed Gromov--Hausdorff topology.
\end{theorem}

In fact, it turns out that for fixed $\theta>0$ all spaces
$(C(\theta),d_t)$ for $t>0$ are isometric up to a multiplicative
constant. An isometry is induced by the radial dilation $x\in C(\theta)\mapsto
t^{-1/2} x$. Our second result shows that for the cone of angle
$\pi/2$ the behavior of the singularity is even worse (see Theorem \ref{thm:singular2} and Proposition
  \ref{prop:convergence2} below).

\begin{theorem}\label{thm:main-cone2}
  Let $C(\pi/2)$ be the two dimensional Euclidean cone of angle
  $\pi/2$ and $d$ its distance. For every $t>0$, the distance $d_t$
  has a conic singularity of angle zero at the apex.  

  As $t$ goes to zero, the metric space $(C(\pi/2),d_t)$ tends to
  $(C(\pi/2),d)$ pointwise and in the pointed Gromov--Hausdorff
  topology. As $t$ goes to infinity, it tends to $\R^+$ with the Euclidean distance in the pointed Gromov--Hausdorff sense.
\end{theorem}

The reason why we focus on these two specific cones is that they can be
conveniently represented as quotients of $\R^2$ under rotation by
$\pi$ and $\pi/2$ respectively. It turns out that the convoluted
distance $d_t$ is the length distance induced by the $L^2$
Wasserstein distance between a mixtures of two (respectively four)
rotated copies of Gaussian measures with variance $2t$.

A corollary of the previous theorem is that the space $(C(\pi/2),d_t)$
is not an Alexandrov space even though $C(\pi/2)$ is. In fact, in
Alexandrov spaces a triangle with one angle zero is flat, which is
wrong for $(C(\pi/2),d_t)$. This negative result has to be compared to
positive results by Takatsu \cite{Tak}, where it is shown that the
subspace made of all Gaussian measures in the Wasserstein space over
Euclidean space is an Alexandrov space. Note moreover, that the
Wasserstein space over a non-negatively curved Alexandrov space is
again a non-negatively curved Alexandrov space \cite[Proposition
I.2.10]{S06} and that many subspaces of finite dimensional Alexandrov
spaces are known to be Alexandrov spaces, for instance convex
hypersurfaces in Euclidean spaces or Riemannian manifolds of sectional
curvature bounded below \cite{AKP, Buj, Milka}.

Given the relation of the Gigli--Mantagazza flow with the Ricci flow,
the convergence of $d_t$ to the original cone distance $d$ has to be
compared with the fact that any Euclidean cones of dimension 2 can be
obtained as the backward limit of classical solutions to the Ricci
flow \cite[Chapter 4.5]{CLN}. See also \cite{SchSim, Deru} for related
results in higher dimension.

\medskip

Our second contribution in this paper is an investigation of the
Gigli--Mantegazza flow applied to the first Heisenberg group equipped
with the Carnot-Carath\'eodory distance. The Heisenberg group is one
of the simplest examples of a non trivial Carnot group, i.e a
nilpotent stratified Lie groups with a left-invariant metric on the
first strata, and of a non trivial subRiemannian manifold. These
classes are of course connected: As proved by Bella\" iche
\cite{Bella}, the tangent cones at points of subRiemannian spaces are
Carnot groups. The differentiable structure of the Heisenberg group is
the one of $\R^3$ and the group structure is given in coordinates
$(x,y,u)$ by $(x,y,u).(x',y',u')=(x+x',y+y',u+u'+(1/2)(xy'-x'y))$.

The Carnot--Carath\'eodory distance is obtained by minimizing the
length of curves that are tangent to the 2-dimensional horizontal
subbundle spanned by $X=\partial_x-\frac{y}{2}\partial_u$ and
$Y=\partial_y+\frac{x}{2}\partial_u$. A standard way to approximate
this distance is to consider for $\eps>0$ the Riemannian distance
$d_{\Rm(\eps)}$ obtained by considering $X,Y,\eps \partial_u$ as an
orthonormal frame. In fact, this penalization principle permits to see
any subRiemannian manifold as a limit of Riemannian manifolds. Note
that $(\H,d_{cc})$ does not satisfy a generalized lower Ricci
curvature bound in the sense of the RCD condition. Therefore we
slightly generalize the construction in \cite{GM} and obtain the
following result (see Theorem
  \ref{thm:Heisenberg-main} and Proposition \ref{prop:approximation-heisenberg}
  below).

\begin{theorem}\label{thm:main-H}
  Let $(\H,d_{cc})$ be the first Heisenberg group equipped with the
  Carnot--Carath\'eodory distance. For $t>0$, the convoluted distance
  $d_t$ coincides with $Kd_{\Rm(\kappa\sqrt{t})}$, for some constants
  $K,\kappa$ satisfying $K\geq 2$ and $K/\kappa<\sqrt{2}$.

  As $t$ goes to zero the distance $d_t$ converges to $Kd_{cc}$
  pointwise. In the pointed Gromov--Hausdorff topology the space
  $(\H,d_t)$ converges to $(\H,d_{cc})$.
\end{theorem}

The striking part of the theorem is that also non-horizontal curve can
have finite length after lifting them to the Wasserstein space built
from $d_{cc}$ via the heat kernel and thus $d_t$ becomes a Riemannian
distance. We believe that this behavior also holds for more general
contact manifolds. However, let us stress the fact that even for the
Heisenberg group the distance $d_t$ does not converge pointwise to
$d_{cc}$ as $t$ goes to zero. Convergence in pointed Gromov--Hausdorff
sense only holds due to the high amount of symmetry of the space, in
particular, due to the fact that the dilation $(x,y,u)\mapsto
(Kx,Ky,K^2u)$ is an isometry between $(\H,K d_{cc})$ and
$(\H,d_{cc})$. The Gromov--Hausdorff convergence probably does not
hold for generic contact manifolds of dimension 3 with a subRiemannian
metric on the nonholonomic contact distribution. Finally, note that
also the Heisenberg group can be obtained as a backward limit of
classical solution to the Ricci flow as was shown by Cao and
Saloff-Coste \cite{CSc}.

\medskip

Three sections follow this introduction. The next section contains the
construction of the convoluted distance $d_t$ in a general setting. As
a first example we discuss the case of normed spaces. In Section
\ref{sec:cone} we establish our results on the Euclidean cones
$C(\pi)$ and $C(\pi/2)$. Section \ref{sec:subriem} is devoted to the
Heisenberg group.

\subsection*{Acknowledgements}
The authors would like to thank Michel Bonnefont, Thomas Richard and
Andr\'e Schlichting for stimulating discussions on this work and
related topics. Part of this work was accomplished while the authors
were enjoying the hospitality of the Hausdorff Research Institute for
Mathematics in Bonn during the Junior Trimester Program on Optimal
Transport. They would like to thank HIM for its support and the
inspiring atmosphere.  M.E. gratefully acknowledges support by the
German Research Foundation through the Collaborative Research Center
1060 \emph{The Mathematics of Emergent Effects} and the Hausdorff
Center for Mathematics.  N.J. is partially supported by the Programme
ANR JCJC GMT (ANR 2011 JS01 011 01).

\section{Construction of the flow}
\label{sec:construction}

In this section we present the construction of the convoluted distance
$d_t$ in a general framework. The reason is that the framework of RCD
spaces considered in \cite{GM} (see subsection \ref{sec:RiemRCD}) does not
cover the Heisenberg group. Moreover, unlike in \cite{GM} the spaces of the present paper are non-compact

\subsection{Preliminaries}

Let $(X,d)$ be a Polish metric space. Recall that for $p\geq 1$ a curve
$(\gamma)_{t\in[0,T]}$ in $(X,d)$ is called $p$-absolutely continuous, for
short $\gamma\in \AC^p\big([0,T],(X,d)\big)$, if there exist a
function $m\in L^p(0,T)$ such that for any $0\leq s\leq t\leq T$:
\begin{align*}
d(\gamma(s),\gamma(t)) ~\leq~ \int_s^tm(r)\dd r\;.
\end{align*}
For $p=1$, we may simply call it an absolutely continuous curve. In
this case the metric derivative defined by
\begin{align*}
  |\dot\gamma_s| ~=~ \lim\limits_{h\to0}\frac{d(\gamma_{s+h},\gamma_s)}{h}
\end{align*}
exists for a.e. $s\in(0,T)$ and is the minimal $m$ as above, see
\cite[Thm.~ 1.2.1]{AGS08}. Lipschitz curves with respect to a distance
$d$ are called $d$-Lipschitz curves, they are locally $p$-absolutely
continuous for every $p\geq 1$.

We denote by $\cP(X)$ the set of Borel probability measures. The
subset of measures with finite second moment, i.e.~satisfying 
\begin{align*}
  \int d(x_0,x)^2\dd\mu(x)<\infty
\end{align*}
for some, hence any $x_0\in X$ will be denoted by $\cP_2(X)$. Given
$\mu,\nu\in\cP_2(X)$ their $L^2$-Wasserstein distance is defined by
\begin{align*}
  W(\mu,\nu) = \inf\limits_\pi \sqrt{\int d(x,y)^2\dd \pi(x,y)}\;,
\end{align*}
where the infimum is taken over all couplings $\pi$ of $\mu$ and
$\nu$. Recall that $\big(\cP_2(X),W\big)$ is again a Polish metric
space. Sometimes we will write $W_{X}$ or $W_{(X,d)}$ to avoid
confusion about the underlying metric space $(X,d)$.

\subsection{Construction of the flow}

Recall that $(X,d)$ is a metric Polish space. Let us assume in addition
that it is proper, i.e.~closed balls are compact, and that it is a length space, i.e.~
we have
\begin{align*}
  d(x,y) = \inf\limits_{\gamma}\int_0^T|\dot\gamma_s|\dd s\;,
\end{align*}
where the infimum is taken over all absolutely continuous curves
$\gamma$ connecting $x$ to $y$. Notice that $(X,d)$ is in fact geodesic,
i.e.~each pair of points can be joint by a curve whose length equals
$d(x,y)$.

The construction is based on a family of maps from $X$ to $\cP_2(X)$ satisfying some properties that we list now. One should keep in mind that in the examples coming later the points are mapped to heat kernel measures.

\begin{assumption}\label{ass:iota}
  There exists a family $(\iota_t)_{t\geq0}$ of maps
  $\iota_t:X\to\cP_2(X)$ with the following properties:
  \begin{itemize}
  \item $\iota_0(x)=\delta_x$ for all $x\in X$,
  \item $\iota_t$ is injective for all $t\geq0$,
  \item $\iota_t$ is Lipschitz, more precisely, there exist constants $C_t>0$ such
  that
  \begin{align}\label{eq:i-Lip}
    W\big(\iota_t(x),\iota_t(y)\big) ~\leq~ C_t d(x,y)\quad \forall x,y\in X\;,
  \end{align}
  and $t\mapsto C_t$ is locally bounded from above,
  \item the curve $[0,\infty)\ni t\mapsto \iota_t(x)$
  is continuous with respect to $W$ for all $x\in X$.
  \end{itemize}
 \end{assumption}

We introduce a new family of distance functions $\tilde d_t:X\times
X\to[0,\infty)$ for $t\geq 0$ given by
\begin{align*}
  \tilde d_t(x,y) ~=~ W\big(\iota_t(x),\iota_t(y)\big)\;.
\end{align*}
As $W$ is a distance it follows from the injectivity of $\iota_t$ that
$\tilde d_t$ is also a distance. It is the \emph{chord} distance
induced by the embedding $\iota_t$. The main object of study here will
be the corresponding \emph{arc} distance, i.e.~the length distance
induced by $\tilde d_t$, denoted by $d_t$. More precisely, we define
for $t\geq 0$ and $x,y\in X$:
\begin{align}\label{eq:def-dt}
  d_t(x,y) ~=~ \inf\limits_\gamma \int_0^T|\dot\gamma_s|_t \dd s\;,
\end{align}
where the infimum is taken over all curves
$\gamma\in\AC\big([0,T];(X,\tilde d_t)\big)$ such that
$\gamma_0=x,\gamma_T=y$ and $|\dot\gamma_s|_t$ denotes the metric
derivative with respect to $\tilde d_t$. Note that \eqref{eq:i-Lip}
implies that
\begin{align}\label{eq:d-t-Lip}
  d_t(x,y) ~\leq~ C_t d(x,y) \qquad\forall x,y\in X\;.
\end{align}
Indeed, for any curve $(\gamma_s)_s$ that is absolutely continuous
with respect to~$d$ its metric derivative with respect to~$d$ is
bounded above as $|\dot\gamma_s|_t\leq C_t|\dot \gamma_s|$. The claim
then follows by integrating in $s$ and taking the infimum over all
such curves $(\gamma_s)_s$ noting that they are also absolutely
continuous with respect to~$\tilde d_t$ and that $(X,d)$ is a length
space.

\begin{remark}\label{rem:def-dt} 
  This construction is slightly different from the one in \cite{GM},
  where the infimum in the definition of $d_t$ is taken over $\gamma$
  in $\AC\big([0,T];(X,d)\big)$ which is a subset of
  $\AC\big([0,T];(X,\tilde d_t)\big)$ by the Lipschitz assumption
  \eqref{eq:i-Lip}. Allowing curves in the latter larger class will be
  crucial when applying the construction in the case of the Heisenberg
  group in Section \ref{sec:subriem}. In the case of the Euclidean
  cones $C(\pi), C(\pi/2)$ discussed in Section \ref{sec:cone}, we
  show in Lemma \ref{lem:biLip-cone} that the infima over both classes
  of curves agree so that we are consistent with the construction in
  \cite{GM}.
\end{remark}

\begin{remark}\label{rem:def-dt2}
  Note that the value of the infimum in \eqref{eq:def-dt} does not
  change, if we restrict the infimum to $\tilde d_t$-Lipschitz
  curves. Indeed, the right hand side of \eqref{eq:def-dt} is
  invariant by reparametrizreparametrizationation and every absolutely continuous curve
  can be reparametrized as a Lipschitz curve, see for instance
  \cite[Lem. 1.1.4]{AGS08}.
\end{remark}

We can reformulate the definition of $d_t$ as follows. Given an
absolutely continuous curve $(\gamma_s)_{s\in[0,T]}$ in
$(X,\tilde d_t)$ we obtain an absolutely continuous curve
$(\mu_{\gamma_s}^t)_{s\in[0,T]}$ in $\big(\cP_2(X),W\big)$ by setting
$\mu^t_{\gamma_s}=\iota_t(\gamma_s)$. Then we have
\begin{align*}
 d_t(x,y)~=~\inf\limits_\gamma\int_0^T|\dot\mu^t_{\gamma_s}|\dd s\;,
\end{align*}
where $|\dot\mu^t_{\gamma_s}|$ denotes the metric derivative with
respect to $W$. Another equivalent formulation is
\begin{align}\label{partition}
d_t(x,y)=\inf\sup\sum_{i=0}^{N-1}\tilde{d_t}(\gamma_{s_i},\gamma_{s_{i+1}})= \inf\sup\sum_{i=0}^{N-1}W(\mu^t_{\gamma_{s_i}},\mu^t_{\gamma_{s_{i+1}}}),
\end{align}
the supremum being taken over all partitions
$0=s_0<s_1<\cdots<s_N=1$ and the infimum over all continuous curves
$(\gamma_s)_{s\in[0,1]}$ connecting $x$ to $y$.

In this general setup we have the following continuity properties.

\begin{proposition}\label{prop:convergence}
  For all $x,y\in X$, the curve $[0,\infty)\ni t\mapsto
  \tilde d_t(x,y)$ is continuous and the curve $ t\mapsto
  (X,\tilde d_t)$ is continuous with respect to the pointed
  Gromov--Hausdorff convergence. 
  Moreover, assume in addition to Assumption \ref{ass:iota} that
  bounded sets in $(X,\tilde d_t)$ are bounded in $(X,d)$. Then the
  distances $\tilde d_t$ and $d_t$ induce the same topology as the
  original distance $d$.
\end{proposition}
\begin{proof}
  We first prove the convergence statement. Let $(t_n)_n$ converge
  to $t$. As an immediate consequence of Assumption \ref{ass:iota} we
  have that $\tilde d_{t_n}(x,y)\to \tilde d_t(x,y)$ for fixed $x,y\in
  X$. Moreover, by \eqref{eq:i-Lip}, for each compact set $K$ in
  $(X,d)$ the functions $\tilde d_{t_n}(\cdot,\cdot)$ are
  equicontinuous on $K\times K$. Thus, they converge uniformly to
  $\tilde d_t(\cdot,\cdot)$. This readily yields the convergence of
  $(X,\tilde d_{t_n})$ to $(X,\tilde d_t)$ in the pointed
  Gromov--Hausdorff sense.
  Now, we turn to the second statement. First, recall from
  \eqref{eq:d-t-Lip} that $\tilde d_t\leq d_t\leq C_t d$. Thus, it
  suffices to show that for any sequence $(x_n)_n$, and element $x$ of $X$ with $\tilde
  d_t(x,x_n)\to0$ as $n\to \infty$ we also have that
  $d(x_n,x)\to0$. By assumption, the sequence $x_n$ is bounded in
  $(X,d)$. Thus, up to taking a subsequence we can assume that
  $d(x_n,x')\to 0$ for some $x'\in X$. Hence, also $\tilde
  d_t(x_n,x')\to0$ and we infer that $x'=x$. This being independent of
  the subsequence chosen, we conclude that the full sequence $x_n$
  converges to $x$ in $(X,d)$.
\end{proof}

\begin{remark}\label{rem:dt-continuity}
  We proved the continuity of the map $t\mapsto \tilde d_t(x,y)$. The
  continuity of $t\mapsto d_t(x,y)$ fails for the Heisenberg group at $t=0$
   as we will see in Section \ref{sec:subriem}. This is in contrast to \cite[Thm.~5.18]{GM}
  where right-continuity of this map is shown. Note however, that the
  Heisenberg group does not satisfy the RCD condition and our
  construction is slightly different in this case, see Remark
  \ref{rem:def-dt}.
\end{remark}

\subsection{Riemannian manifolds and RCD spaces}
\label{sec:RiemRCD}

In \cite{GM} the preceding construction has been introduced and
studied in the case where $(X,d)$ is a Riemannian manifold or more
generally a metric measure spaces satisfying the \emph{Riemannian
  curvature-dimension condition} for some curvature parameter $K \in
\R$, denoted by RCD$(K,\infty)$. For short we call
  such spaces RCD spaces. In both cases the embedding $\iota_t$ is
constructed using the heat kernel. Let us briefly recall the main
results in \cite{GM}.

\medskip

Let $(X,g)$ be a smooth compact and connected Riemannian manifold with
metric tensor $g$ and let $d$ and $\vol$ be the associated Riemannian
distance and volume measure.  One can define a map
$\iota_t:X\to\cP_2(X)$ be setting $\iota_t(x)=\nu^t_x$, where
$\nu^t_x(\mathrm{d} y)=p_t(x,y)\vol(\mathrm{d} y)$ is the heat kernel
measure, i.e.~ $p_t(\cdot,\cdot)$ is the fundamental solution to the
heat equation on $X$. It can be verified that Assumption
\ref{ass:iota} and Proposition \ref{prop:convergence} hold in this
case.

Gigli and Mantegazza prove that the distances $d_t$ are induced by a
family of smooth metric tensors $(g_t)_{t\geq0}$ and that this flow of
tensors is initially tangent to the Ricci flow \cite[Prop.~ 3.5,Thm.~
4.6]{GM}. More precisely, for every geodesic
$(\gamma_s)_{s\in[0,1]}$ with respect to $g=g_0$:
\begin{align*}
  \ddt g_t(\dot\gamma_s,\dot\gamma_s)\big|_{t=0} ~=~ \Ric(\dot\gamma_s,\dot\gamma_s)\quad \text{for almost every } s\in(0,1)\;,
\end{align*}
where $\Ric$ denotes the Ricci tensor of $g$. Gigli and Mantegazza
then generalize the construction for the initial data being a metric
measure space satisfying the RCD$(K,\infty)$. Since we do not work in
this general setting, we will describe it only briefly. For more
details on RCD spaces we refer to \cite{AGS_invent, AGS_duke}.

Roughly speaking, RCD spaces form a natural class of metric measure
spaces that can be equipped with a canonical notion of Laplace operator
and a well behaved associated heat kernel. The RCD$(K,\infty)$ is a reinforcement of the
curvature-dimension condition CD$(K,\infty)$ introduced by
Lott--Villani and Sturm \cite{LV09,S06} as a synthetic definition of a
lower bound $K$ on the Ricci curvature for a metric measure space $(X,d,m)$.
The condition CD$(K,\infty)$ asks for the relative entropy
\begin{align*}
  \Ent(\mu) = \int\rho\log\rho\dd m\;,\quad\text{for }\mu=\rho m\in\cP_2(X) 
\end{align*}
to be $K$-convex along Wasserstein geodesics, i.e.~
\begin{align*}
  \Ent(\mu_s)\leq (1-s)\Ent(\mu_0) + s \Ent(\mu_1) -\frac{K}{2}s(1-s)W(\mu_0,\mu_1)^2\;.
\end{align*}
The RCD$(K,\infty)$ condition requires in addition that the `heat
flow' obtained as the Wasserstein gradient flow of the entropy in the
spirit of Otto \cite{O01} is linear. This excludes e.g.~ Finslerian
geometries. It is a deep insight that the two requirements can be
encoded simultaneously in the following property (which we take as a
definition of RCD spaces for the purpose of this paper).

\begin{theorem}[Definition of the RCD spaces through the EVI
  {\cite[Thm. 5.1]{AGS_duke}}]
  Let $K$ be a real number. The metric measure space $(X,d,m)$
  satisfies the Riemannian curvature-dimension condition
  RCD$(K,\infty)$ if and only if for every $\mu\in\cP_2(X)$ there
  exist an absolutely continuous curve $(\mu_t)_{t\geq0}$ in
  $(\cP_2(X),W)$ starting from $\mu$ in the sense that
  $W_2(\mu,\mu_t)\to0$ as $t\to0$ and solving the
  \emph{Evolution Variational Inequality} (in short EVI)
    of parameter $K$, i.e.~for all $\nu\in\cP_2(X)$
  such that $\Ent(\chi|m)<\infty$ and a.e. $t>0$:
  \begin{align*}
    \ddt \frac12 W(\mu_t,\chi)^2 + \frac{K}{2}W(\mu_t,\chi)^2 \leq \Ent(\chi) -\Ent(\mu_t)\;.
  \end{align*}
\end{theorem}
In fact, the solution $\mu_t$ to the EVI is unique and, putting
$H_t\mu=\mu_t$, one obtains a linear semigroup on $\cP_2(X)$ which is
called the heat flow (acting on measures) in $X$. The construction in
\cite{GM} then proceeds as presented in Section
\ref{sec:construction} by choosing the map $\iota_t:X\to\cP_2(X)$ to
be $\iota_t(x)=H_t\delta_x$. A natural example of RCD spaces are
Euclidean cones, see \cite{Ket15}.

\subsection{Normed spaces}
\label{sec:finsler}

For an example that can be studied rapidly and is rather different let us consider the flow for $\R^n$ equipped with a norm $\norm{\cdot}$. Indeed, the metric measure space $(\R^n,\norm{\cdot},\Leb)$ satisfies the condition
CD$(0,\infty)$ but does not satisfy RCD$(0,\infty)$ unless
$\norm{\cdot}$ is induced by an inner product. It is possible to
consider in this setting a non-linear heat equation, driven by a
non-linear Laplace operator, see \cite{OhSt_cpam} for the a study in
the much more general setting of Finsler manifolds. However, for a
non-Hilbert norm there is no canonical choice of a heat kernel, i.e.~a
solution starting from a Dirac mass since contraction of the heat flow
fails \cite{OhSt_arma}. Note however, that a particular solution is
given by the appealing formula \cite[Example 4.3]{OhSt_arma}
\begin{align*}
   f_t(x)=\frac{C}{4\pi t}\exp\left(-\frac{\norm{x}^2}{4t}\right)\;,
\end{align*}
where $C$ is a normalization constant. Hence a choice satisfying Assumption \ref{ass:iota} is $\iota_t(x)=f_t(\cdot- x)\Leb$. Any other reasonable choice should be translation invariant. Let us show that in this case the
distance $d_t$ coincides with the original one, i.e.~
$d_t(x,y)=\norm{x-y}$. Indeed, consider $\iota_t:x\mapsto
(\tau_x)_\#\nu_t$ where $\nu_t\in\cP_2(\R^d,\norm{\cdot})$ is a
measure and $\tau_x$ the translation by $x$. It is easily checked
using Jensen's inequality on the convex function $(u,v)\mapsto
\norm{u-v}^2$ that
$W_{(\R^n,\norm{\cdot})}\big(\iota_t(x),\iota_t(y)\big)=\norm{x-y}$. The
translation $\tau_{y-x}$ is an optimal map, in other words
$(\tau_x,\tau_y)_\#\nu_t$ is an optimal coupling. Since the original
distance was already a length distance we find $d_t(x,y)=\tilde
d_t(x,y)=\norm{x-y}$. Hence the flow leaves the space invariant and
does not regularize it to a Riemannian manifold.

\begin{remark} We stress that the approximation of some normed spaces
  by Riemannian manifolds is possible by using periodic Riemannian
  metrics with a period diameter going to zero. Consider for instance
  the sequence $(\R^n,k^{-1}d_g)_{k\geq 1}$ where $d_g$ is a fixed periodic Riemannian
  distance. It converges to $\R^n$ equipped with its ``stable norm''
  as defined for instance in \cite[section 8.5.2]{BBI}.  It is not
  clear whether any norm may be attained in this way and this question
  is related to the notorious open problem of characterizing the
  stable norms \cite{BIK}. Finally, note that it is impossible to
  approximate a non-Hilbertian normed space in Gromov Hausdorff
  topology by Riemannian manifolds with non-negative Ricci
  curvature. This is because any such limit metric measure space that
  contains a line has to split as a product of $\R$ and another metric
  measure space by the splitting theorem for Ricci limit spaces
  established by Cheeger and Colding \cite{CheCol}, see also
  \cite[Conclusions and open problems]{Vil09}. This argument also
  applies to the Heisenberg group. Moreover it is proven in
  \cite{Jui09} that $(\H,d_{cc})$ also cannot be approximated by a sequence of
  Riemannian manifolds with \emph{any} uniform lower bound on the
  Ricci curvature.
\end{remark}

\section{Gigli--Mantegazza flow starting from a cone}
\label{sec:cone}

In this section we will analyse the construction in the case where the initial datum is an
Euclidean cone. More precisely, we will consider the cones of angle
$\pi$ and $\pi/2$. We will show that for all times $t$ the resulting
metric $d_t$ retains a warped product form in both cases. In the first
case, it has a conic singularity of angle $\sqrt{2}\pi$ at the apex
for all $t$. In the second case, the asymptotic angle at the apex is
zero for all $t$. Thus in these natural examples, the flow does not
smoothen out the singularity.

In Sections \ref{sec:prelim-cone} to \ref{sec:persistence} we will
present the case of the cone of angle $\pi$ in detail. For the cone of
angle $\pi/2$ we will state the main results in Section
\ref{sec:small-cone} and omit part of the proofs, since the arguments
are very similar.

\subsection{Preliminaries}
\label{sec:prelim-cone}

We will first recall basic properties of Euclidean cones and give an
explicit representation of the heat kernel on the cone of angle $\pi$
in the sense of RCD spaces. Moreover, we will exhibit a
convenient way to calculate Wasserstein distances in the cone, via a
lifting procedure from the cone to $\R^2$.

\subsubsection{Euclidean cones and optimal transport}
The Euclidean cone $C(\theta)$ with angle $\theta\in[0,2\pi]$ is
defined as the quotient 
\begin{align*}
  C(\theta)~=~\Big([0,\infty)\times[0,\theta]\Big)\Big/\sim\;,
\end{align*}
where we write $(r,\alpha)\sim(s,\beta)$ if and only if $r=s=0$ or
$|\alpha-\beta|\in\{0,\theta\}$. The cone distance $d$ is given by
\begin{align*}
  d\big(r,\alpha),(s,\beta)\big) ~=~ \sqrt{r^2+s^2-2rs\cos\big(\min\big(|\alpha-\beta|,\theta-|\alpha-\beta|\big)\big)}\;,
\end{align*}
which is well defined on the quotient. Note that the
  cone without the apex, i.e $C(\theta)\setminus\{o\}$, where $o$ is
  the equivalence class of $(0,0)$, is an open Riemannian manifold
with the metric tensor $(\mathrm{d} r)^2+r^2(\mathrm{d}\alpha)^2$. Its
geometry is locally Euclidean. The associated Riemannian distance is
the cone distance and the distance on the full cone $C(\theta)$ is its
metric completion.

We will be concerned in particular with the cone of angle $\pi$. In
this case we have the alternative characterization as the quotient
\begin{align*}
  C(\pi)=\R^2\big/\sigma\;,
\end{align*}
where the map $\sigma:\R^2\to\R^2$ is the reflection at the origin,
i.e. $\sigma(x)=-x$. Let us denote by $P:\R^2\to C(\pi)$ the canonical
projection. Then the cone distance between $p,q\in C(\pi)$ can be
written as
\begin{align*}
  d(p,q) = \min\big(|x-y|,|x+y|\big)\;,
\end{align*}
where $x,y\in \R^2$ are such that $P(x)=p,P(y)=q$. The Hausdorff
measure on $C(\pi)$ is given as $m=\frac12 P_\#\Leb$, where $\Leb$
denotes the Lebesgue measure on $\R^2$.

Now, we show how to calculate efficiently Wasserstein distance in the
cone $C(\pi)$. We will denote by $W_{\R^2}$ and $W_{C(\pi)}$ the
$L^2$ transport distances on $\R^2$ and $C(\pi)$ built from the
Euclidean distance and the cone distance $d$ respectively. If no
confusion can arise we shall simply write $W$.

Let us introduce the set of measures on $\R^2$ with finite second
moment, that are symmetric with respect to the origin. We set
\begin{align*}
  \cP_2^{\mathrm{sym}}(\R^2) = \{\mu\in\cP_2(\R^2)~:~\sigma_{\#}\mu=\mu \}\;.
\end{align*}
Note that given a measure $\nu\in\cP_2(C(\pi))$ there exists a unique
measure $L(\nu)\in\cP_2^{\mathrm{sym}}(\R^2)$ such that
$P_\# L(\nu)=\nu$. We call $L(\nu)$ the \emph{symmetric lift} of
$\nu$.

We have the following useful fact.

\begin{lemma}\label{lem:lifting}
 For any two measures $\mu,\nu\in\cP_2(C(\pi))$ it holds
  \begin{align*}
    W_{C(\pi)}(\mu,\nu) = W_{\R^2}(L(\mu),L(\nu))\;.
  \end{align*}
  In other words, the mapping $\cP_2^{\mathrm{sym}}(\R^2)\to \cP_2(C(\pi))$,
  $\mu\mapsto P_{\#}\mu$ is an isometry. Moreover, for any two
  measures $\mu,\nu\in\cP_2(\R^2)$ we have
  \begin{align*}
    W_{C(\pi)}(P_\#\mu,P_\#\nu) \leq W_{\R^2}(\mu,\nu)\;.
  \end{align*}
\end{lemma}

\begin{proof}
  Let us first prove the second statement. Let $\mu,\nu\in\cP(\R^2)$
  and $\pi$ a transport plan between $\mu$ and $\nu$. Define a
  transport plan $\bar\pi$ between $P_\#\mu$ and $P_\#\nu$ by setting
  $\bar{\pi}=(P\otimes P)_\#\pi$. Therefore,
  \begin{align}\label{eq:simple}
    \int |y-x|^2\dd \pi(x,y)\geq \int d(P(y),P(x))^2\dd
    \pi(x,y)
   =\int d^2\dd \bar\pi\;.
  \end{align}
  Taking the infimum over $\pi$, we get the second statement. We turn now to the first statement. Let $\mu,\nu\in\cP_2^{\mathrm{sym}}(\R^2)$ and let $\bar{\pi}$ be a
  transport plan between $P_\#\mu$ and $P_\#\nu$. We can find a
  measurable map $Q:C(\pi)\times C(\pi)\to\R^2\times \R^2$ such that
  $(P\otimes P)\circ Q=\mathrm{Id}$ and $|x-y|=d(p,q)$ for
  $Q(p,q)=(x,y)$. These properties also hold for $-Q$ that we note $Q^-$. The
  marginals of the transport plan
  $\pi=\frac12(Q_\#\bar\pi+Q^-_\#\bar\pi)$ are symmetric, hence they
  coincide with $\mu$ and $\nu$. Moreover $\pi$ is concentrated on the
  set $\{(x,y)\in\R^2\times\R^2,\,d(P(x),P(y))=|y-x|\}$ so that we
  have equality in \eqref{eq:simple}. Taking the infimum over
  $\bar\pi$ and taking into account the second statement, we obtain
  the first statement.
\end{proof}

\subsubsection{RCD structure and the heat kernel}
Here we verify that the cone $C(\pi)$ fits into the framework of RCD
spaces considered in \cite{GM} and we give an explicit
description of the heat kernel in this case.

Indeed, the metric measure space $(C(\pi),d,m)$ satisfies the condition
RCD$(0,\infty)$ as proven for instance in \cite[Thm. 1.1]{Ket15}. In
order to identify the heat semigroup $H_t$ acting on measures and the
heat kernel $H_t\delta_x$ in this example, it is sufficient to exhibit
an explicit solution to the Evolution Variational Inequality using
\cite[Thm. 5.1]{AGS_duke}, see Section \ref{sec:RiemRCD}. This will be
done again via the lifting to $\R^2$.

We denote by $\gamma^t_x$ the Gaussian measure with variance $2t$
centered at $x\in\R^2$:
\begin{align*}
  \gamma^t_x(\mathrm{d} y)=\frac{1}{4\pi t}\exp\left(-\frac{|y-x|^2}{4t}\right)\dd y\;.
\end{align*}
The heat semigroup in $\R^2$ acting on measures is denoted by
$H^{\R^2}_t$. More precisely, for any $\mu\in\cP_2(\R^2)$ we set
$H^{\R^2}_t\mu(\mathrm{d} x)=\int \gamma^t_y(\mathrm{d} x) \dd\mu(y)$.

Now, put $\nu^t_p=P_\#(\gamma^t_x)$ where $x$ is such that
$P(x)=p$. We define a semigroup $H^{C(\pi)}_t$ acting on $\cP(C(\pi))$ via
\begin{align*}
H^{C(\pi)}_t\mu(\mathrm{d} q) = \int \nu^t_p(\mathrm{d} q) \dd\mu(p)\;.
\end{align*}
Note that we have $H^{C(\pi)}_t= P_\#\circ H^{\R^2}_t\circ L$.

\begin{lemma}[Evolution Variational Inequality]\label{lem:evi}
  For every $\mu,\chi\in \cP_2(C(\pi))$ such that
  $\Ent(\chi)<\infty$ and every $t\geq 0$ we have
  \begin{align*}
    \frac12 W_{C(\pi)}^2(H^{C(\pi)}_t\mu,\chi)-\frac12 W_{C(\pi)}^2(\mu,\chi) \leq t\big[\Ent(\chi)-\Ent(H^{C(\pi)}_t\mu)\big]\;.
  \end{align*}
\end{lemma}

\begin{proof}
  Let $L(\mu),L(\chi)\in \cP_2^\sigma(\R^2)$ be the lifts of
  $\mu,\chi$. Note that $H^{\R^2}_tL(\mu)$ is the symmetric lift of
  $H_t^{C(\pi)}\mu$. Since $H^{\R^2}_t$ satisfies the Evolution Variational
  Inequality, see e.g.~\cite[Thm. 11.2.5]{AGS08}, we find
  \begin{align*}
    \frac12 W^2_{\R^2}(H^{\R^2}_tL(\mu),L(\chi))-\frac12 W^2_{\R^2}(L(\mu),L(\chi))\\
   \leq t\big[\Ent(L(\chi))-\Ent(H^{\R^2}_tL(\mu))\big]\;.
  \end{align*}
  Observing that $\Ent(L(\mu))=\Ent(\mu)$ for any
  $\mu\in\cP_2(C(\pi))$ and its symmetric lift $L(\mu)$ and using
  Lemma \ref{lem:lifting}, this immediately yields the claim.
\end{proof}

In view of \cite[Thm. 5.1]{AGS_duke}, this shows again that
$(C(\pi),d,m)$ satisfies RCD$(0,\infty)$ and that $H^{C(\pi)}_t$ is
the associated heat semigroup. In particular,
$\nu^t_p=H^{C(\pi)}_t\delta_p$ is the heat kernel at time $t$ centered
at $p$.

We finish this section by noting the following contraction property of
the heat flow:
\begin{align}\label{eq:contractivity}
  W_{C(\pi)}(\nu^t_p,\nu^t_q) \leq d(p,q), \qquad\forall p,q\in C(\pi), t\geq0\;.
\end{align}
Indeed, choosing $x,y$ with $P(x)=p,P(y)=q$ and $d(p,q)=|x-y|$, by
Lemma \ref{lem:lifting} and convexity of the squared Wasserstein
distance we have
\begin{align*}
  W_{C(\pi)}(\nu^t_p,\nu^t_q) &= W_{\R^2}\left(\frac12(\gamma^t_x+\gamma^t_{-x}),\frac12(\gamma^t_y+\gamma^t_{-y})\right)\\
                      &\leq W_{\R^2}(\gamma^t_x,\gamma^t_{y}) = |x-y| = d(p,q)\;.
\end{align*}

\subsubsection{Absolutely continuous curves and the continuity equation}
\label{sec:ac-cruves-cont}

We recall the characterization of absolutely continuous curves in
the Wasserstein space of the Euclidean spaces via solutions to the continuity
equation. Moreover, we formulate a convenient estimate on the driving
vector field in the continuity equation.

\begin{proposition}[{\cite[Thm. 8.3.1]{AGS08}}]
  A weakly continuous curve $(\mu_s)_{s\in [0,T]}$ in $\cP_2(\R^n)$ is
  $2$-absolutely continuous with respect to $W$ if and only if there
  exists a Borel family of vector fields $V_s$ with
  $\int_0^T\norm{V_s}_{L^2(\mu_s;\R^n)}^2\dd s<\infty$ such that the
  continuity equation
  \begin{align*}
    \partial_s\mu +\diverg(\mu_sV_s)=0
  \end{align*}
  holds in distribution sense. In this case we have
  $|\dot \mu_s| \leq \norm{V_s}_{L^2(\mu_s;\R^n)}$ for a.e.~$s$. Moreover,
  $V_s$ is uniquely determined for a.e.~$s$ if we require
  \begin{align*}
    V_s\in T_{\mu_s}\cP_2(\R^n) := \overline{ \{ \nabla \psi \ | \ \psi
    \in C^\infty_c(\R^n) \}}^{L^{2}(\mu_s; \R^{n})}
  \end{align*}
  and it holds $|\dot \mu_s|=\norm{V_s}_{L^2(\mu_s;\R^n)}$.
\end{proposition}

The next lemma states a simple condition for existence and uniqueness
of solutions to the continuity equation.

\begin{lemma}\label{lem:H-1vsL2}
  Let $\mu\in \cP_2(\R^n)$ with strictly positive Lebesgue density
  $\rho$ and assume that $\mu$ satisfies the Poincar\'e inequality
  \begin{align*}
    \int |f|^2\dd\mu \leq C \int |\nabla f|^2 \dd\mu\;,
  \end{align*}
  for all $f\in C^\infty_c(\R^n)$ with $\int f\dd\mu =0$. Let $s\in
  L^1(\R^n,\Leb)$ be such that $\int s=0$ and
  \begin{align*}
    \norm{s/\sqrt{\rho}}^2_{L^2} ~=~ \int \frac{s^2(x)}{\rho(x)}\dd x ~<~\infty\;.
  \end{align*}
  Then there exists a unique vector field $V\in T_{\mu}\cP_2(\R^n)$ such that the
  equation
  \begin{align*}
    s+\diverg(\mu V)=0
  \end{align*}
  holds in distribution sense. Moreover, we have
  \begin{align}\label{eq:H-1vsL2}
    \norm{V}_{L^2(\mu;\R^n)}^2 ~=~ \int |V|^2\dd\mu ~\leq~ C \norm{s/\sqrt{\rho}}^2_{L^2}\;.
  \end{align}
\end{lemma}

\begin{proof}
  For any $f\in C^\infty_c(\R^n)$ with $\int f\dd\mu=0$ we deduce from
  the Cauchy--Schwarz and Poincar\'e inequalities that the bilinear $B:f\mapsto \int s f$ satisfies
  \begin{align*}
    B(f)\leq~ \left(\int \frac{s^2}{\rho}\right)^{\frac12}\left(\int f^2 \rho\right)^{\frac12}\leq \norm{s/\sqrt{\rho}}_{L^2} \sqrt{C} \left(\int |\nabla f|^2\dd \mu\right)^{\frac12}\;.
  \end{align*}
  Thus, identifying $f$ with its gradient, the map $B$ can be extended
  to a bounded linear functional on the Hilbert space
  $T_\mu:=T_\mu\cP_2(\R^n)$ equipped with the scalar product
  $$\langle U,W\rangle_{L^2(\mu;\R^n)}=\int U\cdot W \dd\mu\;.$$
  Moreover, the norm of $B$ is bounded by
  $\sqrt{C}\norm{s/\sqrt{\rho}}_{L^2}$. Thus, by the Riesz representation theorem
  there exists a unique vector field $V\in T_\mu$ such that
  $B(W)=\langle V,W\rangle_{L^2(\mu;\R^n)}$ and
  $\norm{V}_{L^2(\mu;\R^n)}\leq\sqrt{C}\norm{s/\sqrt{\rho}}_{L^2}$. In particular, for any 
  $f$ as above we have
  \begin{align*}
    \int sf ~=~ B(f) ~=~ \int V\cdot \nabla f\dd\mu\;.
  \end{align*}
  Thus $V$ is the unique distributional solution to
  $s+\diverg(\mu V)=0$ in $T_\mu$.
\end{proof}

\subsection{Warped structure of the convoluted cone}
\label{sec:warp-cone}
Having identified the heat kernel in Lemma \ref{lem:evi}, we can now analyse in detail the
construction of \cite{GM} in the case of $C(\pi)$. 
Let us define $\iota_t:C(\pi)\to\cP_2(C(\pi))$ via
$\iota_t(p)=\nu^t_p$. This map is obviously injective and by
\eqref{eq:contractivity} satisfies Assumption \ref{ass:iota}. Thus, as
outlined in Section \ref{sec:construction} we introduce
\begin{align*}
  \tilde d_t(p,q)=W_{C(\pi)}(\nu^t_p,\nu^t_q)
\end{align*}
and define $d_t$ to be the associated length distance as in
\eqref{eq:def-dt}. Recall that the use of the the heat equation is
supposed to produce a kind of convolution for metric spaces. The
rotational symmetry of $C(\pi)$ is preserved by this transformation so
that the resulting space will retain a warped structure.

We first give a partial converse to the Lipschitz estimate
\eqref{eq:i-Lip}.

\begin{lemma}\label{lem:biLip}
  For any $t\geq0$ and $r>0$ there exists a constant $C(t,r)$ such
  that for all $p,q\in C(\pi)\setminus B_r$:
  \begin{align}\label{eq:biLip}
    C(t,r) d(p,q)\leq \tilde d_t(p,q)\;,
  \end{align}
  where $B_r=\{p\in C(\pi) : d(o,p)\leq r\}$.
\end{lemma}

In particular, in view of Proposition \ref{prop:convergence} this shows
that $\tilde d_t$ and $d_t$ induce the same topology as the cone
distance on $C(\pi)$.

\begin{proof}
  Let $x,y\in\R^2$ such that $d(p,q)=|x-y|$. Without restriction we
  can assume that $|x|\leq|y|$ and that $x=(x_1,0)$, $y=(y_1,y_2)$
  with $y_1,y_2\geq 0$. Let $A$ be the line passing through the origin
  at angle $3\pi/8$ with the first coordinate axis and let $\proj_A$
  denote the orthogonal projection onto $L$. Then, setting
  $\mu^t_x=\frac12(\gamma^t_x+\gamma^t_{-x})$, we have
  \begin{align*}
    \tilde d_t(p,q) = W_{\R^2}(\mu^t_x,\mu^t_y) \geq W_{\R^1}((\proj_A)_\#\mu^t_x,(\proj_A)_\#\mu^t_y)
  \end{align*}
  Note that $(\proj_A)_\#\mu^t_x$ is the mixture of two
  one-dimensional Gaussians with variance $2t$ and centers $\pm
  \proj_A x$. Note further that $|\proj_A(x)-\proj_A(y)|\geq
  \cos(3\pi/8) |x-y|$ since the angle of $A$ with $y-x$ is less than
  $3\pi/8$. Thus it suffices to establish the following claim: For any
  $t_0\geq0$ and $r>0$ there exists a constant $C(t_0,r)$ such that
  for all $t\leq t_0$ and $x,y\geq r$:
  \begin{align}\label{eq:biLip2}
    C(t_0,r) |x-y|\leq  W_{\R}\big(\frac12(\gamma^t_x+\gamma^t_{-x}),\frac12(\gamma^t_y+\gamma^t_{-y})\big)\;,
  \end{align}
  where by abuse of notation $\gamma^t_x$ denotes also the
  one-dimensional Gaussian measure with variance $2t$ and center
  $x$. By convexity of $W^2_\R$ the right hand side is decreasing in
  $t$. Thus, by scaling it suffices to consider $t=1$. In dimension 1, the optimal transport plan is known to be the
    monotonic rearrangement. The two measures in \eqref{eq:biLip2} are symmetric so
    that the mass on $\R^+$ is mapped on $\R^+$. Observe that the measure $\gamma^1_x+\gamma^1_{-x}$ restricted to $\R^+$ is distributed as $\omega_\#\gamma^1_x$ where $\omega:x\mapsto|x|$. Hence the right hand side of \eqref{eq:biLip2} is $W_\R(\omega_\#\gamma^1_x,\omega_\#\gamma^1_y)$. Applying Jensen's inequality in the definition of $W_\R$ we see that the distance between the means of these measures is a lower bound. But the mean of $\omega_\#\gamma^1_x$ is
$\int |s| \dd\gamma^1_x(s)=\int |s-x| \dd\gamma^1_0(s)$. As a function of $x\in\R^+$, this is a
strictly convex function with derivative zero at zero on the right and tangent to the first bisector at $+\infty$. In fact the second derivative in distribution sense is $2\gamma_0^1$. The estimate \eqref{eq:biLip2} follows from these remarks together, provided $x,y\geq r$ for some $r>0$.
\end{proof}

Next, we show that in the definition of $d_t$ we can restrict the
infimum to Lipschitz curves with respect to~the cone distance $d$.
\begin{lemma}\label{lem:biLip-cone}
For any $t\geq 0$ and $p,q\in C(\pi)$ we have
\begin{align}\label{eq:def-dt-lip}
  d_t(p,q) ~=~ \inf\left\{\int_0^T|\dot p_s|_t \dd s~\right\}\;,
\end{align}
where the infimum is taken over all $d$-Lipschitz curves $(p_s)_{s\in[0,T]}$ such
that $p_0=p,p_T=q$ and $|\dot p_s|_t$ denotes the metric derivative
with respect to~$\tilde d_t$. If $p\neq o$ and $q\neq o$, one can restrict to the curves supported in $C(\pi)\setminus\{o\}$.
\end{lemma}
Thus the construction of $d_t$ given here is consistent with the general
construction in RCD spaces given in \cite{GM}  (see Remark \ref{rem:def-dt}).

\begin{proof}
  The inequality ``$\leq$'' follows immediately from the fact that any
  $d$-Lipschitz curve is also $\tilde d_t$-Lipschitz by
  \eqref{eq:contractivity}. To see the reverse inequality, first
  recall that by Remark \ref{rem:def-dt2} we can restrict the infimum
  in \eqref{eq:def-dt} to $\tilde d_t$-Lipschitz curves. Then the
  statement follows from Lemma \ref{lem:biLip}. Indeed, given
  $\eps>0$, let $(p_s)_{s\in[0,T]}$ be a $\tilde d_t$-Lipschitz curve
  such that
  \begin{align*}
    \int_0^T|\dot p_s|_t\dd s \leq d_t(p,q) + \eps\;.
  \end{align*}
  Recall that $(p_s)$ is $d$-continuous. If it avoids the origin (and
  thus also a neighborhood around it) then $(p_s)$ is also
  $d$-Lipschitz by \eqref{eq:biLip}. If the curve hits the origin, put
  for sufficiently small $r>0$:
  \begin{align*}
    s_1:=\inf\{s\in[0,T]:p_s\in B_r\}\;,\qquad s_2:=\sup\{s\in[0,T]:p_s\in B_r\}\;.
  \end{align*}
  We can construct a $d$-Lipschitz curve $(\tilde p_s)$ by replacing
  the part $(p_s)_{s\in[s_1,s_2]}$ with a piece of circle connecting
  $p_{s_1}$ to $p_{s_2}$. From \eqref{eq:contractivity} we see that
  the $\tilde d_t$-length of $(\tilde p_s)$ is bounded by
  $d_t(p,q)+\eps+\pi r$. Choosing $r$ sufficiently small and using the
  arbitrariness of $\eps$ we obtain the inequality ``$\geq$'' in
  \eqref{eq:def-dt-lip}.
  
In the case $p=o$ or $q=o$, the ray from or to the apex is a minimizing curve. Note that it is a $d$-Lipschitz curve.
\end{proof}

Let us further observe the particular behavior of the distance under
scaling of space and time.
\begin{lemma}\label{lem:scaling}
  For any $t>0$ and $p,q\in C(\pi)$ we have
  \begin{align}\label{eq:scaling}
    d_t(p,q)=\sqrt{t}\cdot d_1\big(\sqrt{t}^{-1}p,\sqrt{t}^{-1}q\big)\;.
  \end{align}
  Here, for $\lambda\geq 0$ and $p=(r,\alpha)\in C(\pi)$ we set
  $\lambda p=(\lambda r,\alpha)$.
\end{lemma}

\begin{proof}
  It suffices to establish the identity \eqref{eq:scaling} with $d_t$
  replaced by $\tilde d_t$. It then passes easily to the associated
  length distance. Recall that $\tilde d_t (p,q) =
  W_{2,\R^2}(\frac12(\gamma^t_x+\gamma^t_{-x}),\frac12(\gamma^t_y+\gamma^t_{-y})$
  for $x,y$ such that $P(x)=p,P(y)=q$. Introduce the dilation
  $s_\lambda:x\mapsto \lambda x$ and note that
  $\gamma^t_x=(s_{\sqrt t})_\#\gamma^1_{\sqrt{t}^{-1}x}$. Now the
  claim is immediate.
\end{proof}

We have the following result on the metric structure of the convoluted
cone.

\begin{proposition}\label{prop:warp}
  The distance $d_t$ is induced by a metric tensor $g^t$ on the
  open manifold $C(\pi)\setminus\{o\}$ which is of warped product form
  \begin{align}\label{eq:warp1}
    g^t_{(r,\alpha)}(\cdot,\cdot) = R(r/\sqrt{t})\mathrm{d r}^2 + r^2A(r/\sqrt{t})\mathrm{d}\alpha ^2\;,
  \end{align}
  where $R,A:(0,\infty)\to(0,1]$ are bounded
  functions. Moreover, the distance $d_t$ on the full cone is obtained
  for $p_0,p_1\in C(\pi)$ by
  \begin{align}\label{eq:warp}
    d_t\big(p_0,p_1\big) = \inf \int_0^1\sqrt{R(r_s/\sqrt{t})|\dot r_s|^2 + r_s^2A(r_s/\sqrt{t})|\dot \alpha_s|^2}\dd s\;,
  \end{align}
  where the infimum is taken over all Lipschitz curves
  $(p_s)_{s\in[0,1]}$ of $(C(\pi),d)$ connecting $p_0,p_1$ and $|\dot
  r_s|,|\dot \alpha_s|$ denote the metric derivatives of the polar
  coordinates of $p_s$.
\end{proposition}

In \eqref{eq:warp1} $R$ and $A$ stand for radial and angular.

\begin{proof}
  Recall from Section \ref{sec:construction} that
  \begin{align*}
    d_t(p,q)=\inf\limits_{(p_s)} \int_0^T |\dot \nu^t_{p_s}|\dd s\;,
  \end{align*}
  where $\nu^t_p=\iota_t(p)$ and $|\dot \nu^t_{p_s}|$ denotes the metric
  derivative with respect to $W_{C(\pi)}$. We will first use the
  lifting to $\R^2$ and the characterization of the Wasserstein metric
  derivative in terms of solutions to the continuity equation to
  relate $d_t$ to a smooth metric tensor on $\R^2$ and then we will
  push this tensor to the cone to obtain the desired warped structure.

  From Lemma \ref{lem:scaling} we immediately infer that it is
  sufficient to consider $t=1$. For brevity let us set
  $\mu_x=\frac12 \gamma_x +\frac12 \gamma_{-x}$ and let $f_x$ be its
  density, i.e.\ $f_x(y)=\frac12 \eta(y-x) + \frac12 \eta(y+x)$, where
  $\eta(y)=\frac1{4\pi}\exp(-|y|^2/4)$ is the density of the
  2-dimensional Gaussian at time $1$.

  We define a metric tensor $\tilde g$ on $\R^2 $ by setting
  for $x,w\in\R^2$:
  \begin{align*}
    \tilde g_x(v,w)~=~\int \langle V^v_x,V^w_x\rangle \dd\mu_x\;,
  \end{align*}
  where $V_x^w$ is the unique vector field in $T_{\mu_x}\cP(\R^2)$ solving 
  \begin{align}\label{eq:conteq}
    \frac{\dd}{\dd h}\big|_{h=0}f_{x+hw} ~&=~ \frac12 w\cdot \big(\nabla \eta(x+y)-\nabla \eta(x-y)\big) ~=~ -\diverg(\mu_xV_x^w) 
  \end{align}
  given by Lemma \ref{lem:H-1vsL2} (applied to $s=\frac{\dd}{\dd h}
  f_{x+hw}$ and $\mu=\mu_x$). Indeed, by uniqueness, $V_x^w$ depends
  linearly on $w$, hence $g_x(v,w)$ is a bilinear form. 
  
  Now, define a metric tensor $g$ on the open manifold
  $C(\pi)\setminus\{o\}$ by setting for $p=(r,\alpha)\in
  C(\pi)\setminus\{o\}$ and $v,\theta\in\R$:
  \begin{align*}
  g_{p}\big((v,\theta),(v,\theta)\big) = \tilde g_x(w,w)\;,
  \end{align*}
  where
  \begin{align*}
   x=r\binom{\cos\alpha}{\sin\alpha}\;,\quad w = v \binom{\cos\alpha}{\sin\alpha} +\theta r \binom{-\sin\alpha}{\cos\alpha}\;,
  \end{align*}
  and extend via polarization. That $g$ takes the form \eqref{eq:warp1} is a consequence of the fact that $\tilde g_x(w,w)$ is
  invariant under reflecting $w$ at the line passing through the
  origin and $x$, implying that
  $g_p\big((v,\theta),(v,\theta)\big)=g_p\big((v,-\theta),(v,-\theta)\big)$, and 
  the invariance of $\tilde g_x(w,w)$ under simultaneous rotation of $x,w$. Explicitly, we have
  \begin{align}\label{def:AR}
  R(r)=\tilde g_{r(1,0)}\big((1,0),(1,0)\big)\;,\qquad A(r)=\tilde g_{r(1,0)}\big((0,1),(0,1)\big)\;.
  \end{align}
  Let us now prove \eqref{eq:warp}, i.e.~that $d_1$ is induced by the
  tensor $g$. By Lemma \ref{lem:biLip-cone} we have
  \begin{align*}
    d_1(p_0,p_1) = \inf\int_0^1 |\dot p_s|_1 \dd s\;,
  \end{align*}
  where $|\dot p_s|_1$ is the metric derivative of $p_s$ with respect
  to the distance $\tilde d_1$ and the infimum is over Lipschitz
  curves in $(C(\pi),d)$. Let us consider a Lipschitz curve
  $(p_s)_{s\in[0,1]}$ in $(C(\pi),d)$ with polar coordinates
  $(r_s)_s$ and $(\alpha_s)_s$. Let $(x_s)_{s\in[0,1]}$ be a continuous a curve such
  that $P(x_s)=p_s$. By \eqref{eq:contractivity} the curves
  $\nu^1_{p_s}$ and $\mu_{x_s}$ are Lipschitz with respect to
  $W_{C(\pi)}$ and $W_{\R^2}$ respectively and by definition of
  $\tilde d_1$ and Lemma \ref{lem:lifting} we have
  \begin{align*}
    |\dot p_s|_1 = |\dot\nu^1_{p_s}| = |\dot\mu_{x_s}|\;,
  \end{align*}
  where the latter two metric derivatives are calculated with respect
  to $W_{C(\pi)}$ and $W_{\R^2}$ respectively. By the characterization
  of absolutely continuous curves, there exists for a.e.~$s$ a vector
  field $V_s\in T_{\mu_{x_s}}\cP(\R^2)$ such that the continuity equation
  $\partial_s\mu_{x_s} =-\diverg(\mu_{x_s}V_s)$ holds in distribution
  sense. But for a.e.~$s$ the left hand side is given by $\frac12
  w_s\cdot \big(\nabla \eta(\cdot+x_s)-\nabla \eta(\cdot-x_s)\big)$
  with
  $$w_s = \dot r_s \binom{\cos\alpha_s}{\sin\alpha_s} +\dot\alpha_s r_s \binom{-\sin\alpha_s}{\cos\alpha_s}\;.$$
  Hence, the uniqueness statement in Lemma \ref{lem:H-1vsL2} implies
  that for a.e.~$s$ we have $V_s=V_{x_s}^{w_s}$ and thus
  \begin{align*}
    |\dot p_s|_1^2 =  \tilde g_{x_s}(w_s,w_s) = g_{p_s}\big((\dot r_s,\dot \alpha_s),(\dot r_s,\dot \alpha_s)\big)
                = |\dot r_s|^2 R(r_s) + r_s^2 |\dot\alpha_s|^2 A(r_s)\;.
  \end{align*}
  This yields that $d_t(p_0,p_1)$ is given by the right hand side in
  \eqref{eq:warp}.

Finally, we turn to the boundedness of $R$ and $A$. In fact for $w$ a vector of $\R^2$ the vector field $V:y\mapsto\lambda_x(y)w+(1-\lambda_x(y))(-w)$ where $\lambda_x(y)=[\eta(y-x)/(\eta(y-x)+\eta(y+x))]$ satisfies \eqref{eq:conteq} in place of $V^w_x$. It is an element of $L^2(\mu_x;\R^n)$ with norm smaller than or equal to $|w|$. The orthogonal projection on $T_{\mu_x}\cP(\R^2)$ contracts the norm and provides another solution to \eqref{eq:conteq}. According to the uniqueness statement in Lemma \ref{lem:H-1vsL2} it is $V^w_x$. Hence we have proved $\tilde{g}_x(w,w)\leq |w|^2$. It follows that the functions $A$ and $R$ defined in \eqref{def:AR} are bounded from above by $1$.

\end{proof}

\begin{remark}\label{rem:smoothness}
  We believe that the functions $R$ and $A$ in the proposition above
  are smooth, so that $d_t$ would be induced by a smooth metric
  tensor on $C(\pi)\setminus\{o\}$. From the explicit
  expression \eqref{eq:R-explicit} for $R$ given below in the proof of
  Theorem \ref{thm:singular}, it is readily checked $R$ is
  smooth. Proving smoothness for $A$ seems non-trivial due to the
  non-compactness of the cone.
\end{remark}

\begin{proposition}\label{prop:convergence1}
  As $t$ goes to zero, the metric space $(C(\pi),d_t)$ tends to
  $(C(\pi),d)$ pointwise and in the pointed Gromov--Hausdorff
  topology.
\end{proposition}

\begin{proof}
  By construction of $d_t$ and by the contractivity \eqref{eq:contractivity} with \eqref{eq:d-t-Lip} we have the chain of inequalities
  \begin{align*}
    \tilde d_t \leq d_t \leq d\;.
  \end{align*}
  From Proposition \ref{prop:convergence} we already know that $\tilde
  d_t$ pointwise converges to $d$ as $t\to0$ whence the convergence of
  $d_t$ follows. The convergence in pointed Gromov--Hausdorff topology
  follows as in the proof of Proposition \ref{prop:convergence}.
\end{proof}

As an immediate consequence of Proposition \ref{prop:warp}, we deduce
that a $d_t$-minimizing curve connecting the apex $o=(0,0)$ to the point
$(r,0)\in C(\pi)$ is given by the curve
$(sr,0)_{s\in[0,1]}$. Hence the distance of $(r,0)$ from the apex $o$
is
\begin{align}\label{eq:radius}
  d_t\big((r,0),o\big) = \int_0^1r\sqrt{R\big(sr/\sqrt{t}\big)}\dd s = \int_0^r \sqrt{R\left(s/\sqrt{t}\right)}\dd s\;.
\end{align}

\subsection{Persistence of the conic singularity for $C(\pi)$}
\label{sec:persistence}

We will show that the new distance $d_t$ has a conic singularity at
the origin of angle $\sqrt{2}\pi$ independent of $t$. In order to do
so we will compare for small $r$ the distance of a point $p_r=(r,0)$
from the origin to the length of a circle around the origin passing
through $p_r$. 

More precisely, for $r>0$ set $\rho_t(r)=d_t\big((0,r),o\big)$ and define
\begin{align*}
  l_t(r)=\int_0^\pi|\dot p^r_s|_t\dd s\;.
\end{align*}
where the curve $p^r:[0,\pi]\to C(\pi)$ is given by $p^r_s=(r,s)$.

\begin{theorem}\label{thm:singular}
   For each $t>0$ we have
  \begin{align*}
    \lim\limits_{r\to0}\frac{l_t(r)}{\rho_t(r)}=\sqrt{2}\pi\;.
  \end{align*}
  In other words, the angle at the apex $o$ is $\sqrt{2}\pi$. In
  particular, a singularity persists at $o$. With the notation of
  Proposition \ref{prop:warp} we have more precisely $R(r)\sim
  r^2/2$ and $A(r)\sim r^2/4$ as $r\to 0$.
  
  Moreover, $C(\sqrt{2}\pi)$ is both, the tangent space of $(C(\pi),d_t)$ at $o$, and the limit in the pointed Gromov--Hausdorff topology as $t$ goes to infinity.
\end{theorem}

\begin{remark}\label{rem:discont}
  The discontinuity at $t=0$ of the asymptotic angle at $o$ might seem
  intriguing at first in view of the convergence of $d_t$ to the
  original distance $d$ given by Proposition
  \ref{prop:convergence1}. Note however, that the asymptotic angle is in
  a certain sense a first order quantity, while the convergence of
  distances is zero order. Intuitively, the discontinuity can be
  understood from the scaling property \eqref{eq:scaling}. After
  zooming in at scale $r$, the heat kernel measure at a very
  small time $t$ looks like the heat kernel measure at the larger time
  $t/\sqrt{r}$ at the original scale.
\end{remark}

\begin{proof}
  We will calculate $\rho_t$ and $l_t$ asymptotically as
  $r\to 0$.
  
  From Proposition \ref{prop:warp} and \eqref{eq:radius} we have
  \begin{align}\label{eq:rho1}
    \rho_t(r) &= \int_0^1r\sqrt{R\big(sr/\sqrt{t}\big)}\dd s\\\label{eq:l1}
    l_t(r) &= \pi r\sqrt{A\big(r/\sqrt{t}\big)}\;.
  \end{align}
  Thus, it remains to calculate $R$ and $A$. Denote by $f_x$ the density of
  $\mu^1_x=\frac12\gamma^1_x+\frac12\gamma^1_{-x}$. We set $x_r=(r,0)\in
  \R^2$ and recall from the proof of Proposition \ref{prop:warp} that
  \begin{align*}
    R(r) &= ||V_{x_r}^{(1,0)}||_{L^2(\mu^1_{x_r};\R^2)}^2\;,\\
    A(r) &= ||V_{x_r}^{(0,1)}||_{L^2(\mu^1_{x_r};\R^2)}^2\;,   
  \end{align*}
  where the vector field $V_{x}^{w}\in T_{\mu^1_{x}}\cP(\R^2)$ is defined
  uniquely by the continuity equation
\begin{align} \label{eq:ceAR}
    \frac{\dd}{\dd h}\big|_{h=0}f_{x+hw} ~&=~ -\diverg(f_xV_x^w)\;.
  \end{align}
  Note that
  $$f_x(y)=1/2\big(\eta(y_1-x_1)\eta(x_2-y_2)+\eta(x_1+y_1)\eta(x_2+y_2)\big)\;,$$
  where $\eta$ denotes $y\mapsto (4\pi t)^{-1/2}\exp(-y^2/4t)$, the 1-dimensional Gaussian density at time $1$. 
  Let us first concentrate on $R$. Here, we have to solve 
  \begin{align*}
    \frac{\dd}{\dd h}\big|_{h=0}f_{x_r+h(1,0)}(y) &= \eta(y_2)\frac12\big(-\eta'(y_1-r)+\eta'(y_1+r)\big)\\
                                             &= - \mathrm{div}\left(f_{x_r}V_{x_r}^{(1,0)}\right)(y)\;.
  \end{align*}
  It is easily checked that the solution is given by
  \begin{align*}
    V_{x_r}^{(1,0)}(y)=\binom{1}{0}\frac{\eta(y_1-r)-\eta(y_1+r)}{\eta(y_1-r)+\eta(y_1+r)}\;.
  \end{align*}
  which indeed belongs to $T_{\mu^1_{x_r}}\cP(\R^2)$. Thus, we have
  \begin{align}\nonumber
    R(r) &= \int |V_{x_r}^{(1,0)}|^2\dd\mu^1_{x_r}
                  = \frac12 \int\int \eta(y_2)\frac{|\eta(y_1-r)-\eta(y_1+r)|^2}{\eta(y_1-r)+\eta(y_1+r)} \dd y_1\dd y_2\\\label{eq:R-explicit}
                  &=  \frac12 \int \frac{|\eta(y_1-r)-\eta(y_1+r)|^2}{\eta(y_1-r)+\eta(y_1+r)} \dd y_1\;.
  \end{align}
  To determine the asymptotic behavior as $r\to0$, we first note that
  \begin{align}\label{eq:rho2}
    R(r)=\frac12 r^2 \int\frac{|2\eta'(y_1)|^2}{2\eta(y_1)}\dd x + o(r^2)
       = r^2\int \frac{y_1^2}{4}\eta(y_1)\dd y_1 +o(r^2) = \frac{r^2}{2}+o(r^2)\;.
  \end{align}  
  
  Let us turn to calculating $A$. First we note that the left hand  side $\beta^r$ of the continuity
  equation \eqref{eq:ceAR} is
  \begin{align*}
    \beta^r=\frac{\dd}{\dd h}\big|_{h=0}f_{x_r+h(0,1)}(y) &= \eta'(y_2)\frac12\big[-\eta(y_1-r) + \eta(y_1+r)\big]\\
                                &=r \eta'(y_1)\eta'(y_2)+o(r)\;.
  \end{align*} 
  Unfortunately, we can not explicitly solve equation \eqref{eq:ceAR} in $T_{\mu_{x_r}^1}\mathcal{P}(\R^2)$ but we can approximate the
  solution. To this end introduce the function $\psi^r:\R^2\to\R$
  given by
  $$\psi^r(y)=r\frac14 y_1y_2\;.$$
  We calculate $\delta^r(y)$ where $\delta^r:y\mapsto -\diverg\big(f_{x_r}\nabla\psi^r\big)(y)$.
  \begin{align*}
    \delta^r(y)&=~ -\frac{r}{4}\Big[y_2\eta(y_2)\frac12\big(\eta'(y_1+r)+\eta'(y_1-r)\big)\\
        &\qquad\qquad + \eta'(y_2)y_1\frac12\big(\eta(y_1+r)+\eta(y_1-r)\big)\Big]\\
	&=~\frac{r}{4}\Big[(1+1)\eta'(y_2)\big(\eta'(y_1+r)+\eta'(y_1-r)\big)\\
        &\qquad\qquad + r \eta'(y_2)\frac12\big(\eta(y_1+r)-\eta(y_1-r)\big)\Big]\\	
        &=~ \frac{r}2\eta'(y_2)\big(\eta'(y_1+r)+\eta'(y_1-r)\big)\\
        &\qquad\qquad + \frac{r^2}{8}\eta'(y_2)\big(\eta(y_1+r)-\eta(y_1-r)\big)\;,
  \end{align*}
  where in the first line we used several times the identity
  $\eta'(t)=-(t/2) \eta(t)$. Considering an expansion in
  $r$ at $r=0$ we check that as $r\to0$ we have
  \begin{align*}
    \frac1{r^2} \int \frac{|\beta^r -\delta^r|^2}{f_{x_r}} \rightarrow 0\;.
  \end{align*}
  Since the measures $\mu^1_{x_r}$ satisfy the Poincar\'e inequality with
  constant independent of $r$, we deduce by Lemma \ref{lem:H-1vsL2} (applied to $\nabla\psi^r -V_{x_r}^{(0,1)}$ and $\mu^1_{x_r}$)
  that
  \begin{align*}
    \frac{1}{r}\norm{\nabla\psi^r -V_{x_r}^{(0,1)}}_{L^2(\mu^1_{x_r};\R^2)}\rightarrow 0\;.
  \end{align*}
  It thus suffices to calculate
  \begin{align*}
    &||\nabla\psi^r||_{L^2(\mu^1_{x_r};\R^2)}\\
     &=~\frac{r}{4}\left(\int (y_1^2+y_2^2)\eta(y_2)\frac12\big(\eta(y_1+r)+\eta(y_1-r)\big)\dd y_1\dd y_2\right)^{\frac12}\\
     &=~\frac{r}{2} +o(r^2)\;.
  \end{align*}
  Thus $\sqrt{A(r)}=\frac{r}{2}+o(r)$. This together with
  \eqref{eq:rho1}, \eqref{eq:l1} yields
  \begin{align*}
    \rho_t(r) &= \frac{1}{2\sqrt{2t}}r^2 +o(r^2)\;,\\
    l_t(r) &= \pi\frac{1}{2\sqrt{t}}r^2 +o(r^2)\;.
  \end{align*}
  This gives the claim on the limit ratio.

For the last part of the statement let us consider the reparametrization $\mathcal{T}:(r,\theta)\in C(\pi)\mapsto (\rho(r),\theta)\in C(\pi)$ where $\rho$ stands for $\rho_t$ at time $1$. We note $(\bar{r},\bar{\theta})$ the new coordinates. The function $\rho$ is continuously differentiable of positive derivative so that $\mathcal{T}$ is a diffeomorphism outside the apex. A curve $(\gamma_s)_{s\in[0,T]}$ with support on $C(\pi)\setminus \{o\}$ is Lipschitz if and only if $(\mathcal{T}\circ\gamma_s)_s$ is Lipschitz too. Moreover, a change of variable shows how to compute the length on the second curve with the tensor defined by 
$$\bar{R}=1\quad\text{and}\quad\bar{A}(\bar{r})\bar{r}^2=A(\rho^{-1}(\bar{r}))\times\rho^{-1}(\bar{r})^2$$ in place of $R$ and $A$. We proved in Lemma \ref{lem:biLip} that in the minimisation problem \eqref{eq:warp} it is possible to use Lipschitz curves outside the apex, or Lipschitz rays from or to the apex. Both classes of curves are preserved by $\mathcal{T}$ and $\mathcal{T}^{-1}$. Finally similarly as in Lemma \ref{lem:biLip} the infimum of the length in the new coordinates remains the same if it is allowed to test the Lipschitz curves going through $o$. Using the equivalents of $A$ and $\rho$ we obtain $\bar{A}^{1/2}\sim_{\bar{r}\to 0}\sqrt{2}$. Note that the equation $\bar{A}=c$ would corresponds to the metric of $C(c\pi)$. With the new coordinates we easily recognize that the tangent space at zero is $C(\sqrt{2}\pi)$. Together with the time-space scaling of Lemma \ref{lem:scaling} we obtain the same limit space when $t$ goes to infinity. 
\end{proof}

\begin{remark}
  The intuition for finding a good candidate $\nabla\phi$ for the solution to
  \begin{align*}
    \beta^r = -\diverg(f_{x_r}\nabla\phi)
  \end{align*}
  is as follows. Since we are interested only in the limit $r\to0$ we
  expand the continuity equation in $r$. Note that
  $f_{x_r}(y_1,y_2)=\eta(y_1)\eta(y_2)+O(r^2)$. We expand
  $\nabla\phi=\nabla\phi_{(0)} + r\nabla\phi_{(1)} +o(r)$.
  Since $\beta^r=r\eta'(y_1)\eta'(y_2)+o(r)$ we conclude that
  $\nabla\phi_0=0$ and that we must have
  \begin{align*}
    \eta'(y_1)\eta'(y_2) &= -\diverg(\eta\otimes \eta\nabla\phi_{(1)})(y_1,y_2)\\
                           &= -\eta'(y_1)\eta(y_2)\partial_{y_1}\phi_{(1)} -\eta(y_1)\eta'(y_2)\partial_{y_2}\phi_{(1)}\\
                           &\quad -\eta(y_1)\eta(y_2)\Delta\phi_{(1)}(y_1,y_2)
  \end{align*}
 or equivalently, since $\eta'(u)=- (u/2) \eta(u)$,
 \begin{align*}
   y_1y_2/4 - (y_1/2)\partial_{y_1}\phi_{(1)}(y) - (y_2/2)\partial_{y_2}\phi_{(1)}(y) + \Delta\phi_{(1)}(y)=0\;.
 \end{align*}
 A solution to this is given by $\phi_{(1)}(y)=\frac1{4}y_1y_2$. 
\end{remark}

\begin{remark}\label{rem:angle-infinity}

  Unsurprisingly, the asymptotic angle of $\big(C(\pi),d_t\big)$ at
  infinity remains $\pi$ independently of $t$. More precisely, for any $t\geq 0$:
  \begin{align}\label{eq:angle-infinity}
    \lim\limits_{r\to\infty}\frac{l_t(r)}{\rho_t(r)} ~=~ \pi\;.
  \end{align}
  Indeed, in view of \eqref{eq:rho1}, \eqref{eq:l1} we find after a
  change of variables in the integral that
  \begin{align*}
    \frac{l_t(r)}{\rho_t(r)}~=~\frac{\pi \sqrt{A(r/\sqrt{t})}\times(r/\sqrt{t})}{\int_0^{r/\sqrt{t}}\sqrt{R(s)}\dd s}\;.
  \end{align*}
  Thus we have
  \begin{align*}
    \lim\limits_{r\to\infty}\frac{l_t(r)}{\rho_t(r)} ~=~ 
     \lim\limits_{t\to0}\frac{l_t(r)}{\rho_t(r)} ~=~\pi\;.
  \end{align*}
In the last equality we used Proposition \ref{prop:convergence1} for $\rho_t(r)\to r$ and also a representation of the length of type \eqref{partition} together with $\tilde{d}_t\leq d_t\leq d$ for $l_r(t)\to \pi r$.
\end{remark}

\subsection{The cone of angle $\pi/2$}
\label{sec:small-cone}

Like $C(\pi)$, the cone $C(\pi/2)$ admits an alternative characterization as a
quotient of $\R^2$, however, it will be convenient to phrase this in
terms of complex numbers. We have
\begin{align*}
  C(\pi/2)=\C\big/\sigma\;,
\end{align*}
where the map $\sigma: \C \to\C$ is the direct rotation by $\pi/2$,
i.e.~$\sigma(z)=iz$. Let us denote by $P:\C\to C(\pi/2)$ the
canonical projection. Then the cone distance between $p,q\in C(\pi/2)$
can be written as
\begin{align*}
  d(p,q) = \min\Big\{|z-e^{ik\pi/2}z'| : k=0,1,2,3\Big\}\;,
\end{align*}
where $z,z'\in \C$ are such that $P(z)=p,P(z')=q$. The Hausdorff
measure on $C(\pi/2)$ is given as $m=\frac14 P_\#\Leb$, where $\Leb$
denotes the Lebesgue measure on $\C$.

As in section \ref{sec:prelim-cone} we can calculate Wasserstein distances in the cone via
lifting. Given $\nu\in\cP_2(C(\pi/2))$ we denote by $L(\nu)$ the
\emph{symmetric lift} of $\nu$, i.e.~the unique measure in
\begin{align*}
  \cP_2^{\mathrm{sym}}(\C) := \{\mu\in\cP_2(\C)~:~\sigma_{\#}\mu=\mu \}\;.
\end{align*}
such that $P_\# L(\nu)=\nu$. Then, in analogy to Lemma
\ref{lem:lifting}, for any two measures
$\mu,\nu\in\cP_2(C(\pi/2))$ we obtain
\begin{align*}
  W_{C(\pi/2)}(\mu,\nu) = W_{\C}(L(\mu),L(\nu))\;.
\end{align*}
Recall that $\gamma^t_z$ denotes the two-dimensional Gaussian measure
with variance $2t$ centered at $z \in\C$. We set
$\nu^t_p=P_\#(\gamma^t_z)$, where $p=P(z)$. By the obvious analogue of
Lemma \ref{lem:evi}, $\nu^t_p$ is the heat kernel measure on
$C(\pi/2)$ in the sense of RCD spaces. Note that its lift is given by
\begin{align*}
  L(\nu^t_p) = \frac14 \Big[\gamma^t_z + \gamma^t_{iz} + \gamma^t_{-z} + \gamma^t_{-iz}\Big]\;.
\end{align*}
Let $d_t$ be the length
distance associated to $\tilde
d_t(p,q)=W_{C(\pi/2)}(\nu^t_p,\nu^t_q)$. It is found again to satisfy
the scaling relation \eqref{eq:scaling}. Arguing exactly as in
Proposition \ref{prop:warp} and Proposition \ref{prop:convergence1} we
obtain

\begin{proposition}\label{prop:warp2}
  The distance $d_t$ is induced by a metric tensor $g^t$ on the open
  manifold $C(\pi/2)\setminus\{o\}$ which is of warped product
  form
  \begin{align}\label{eq:warp2}
    g^t_{(r,\alpha)}(\cdot,\cdot) = R(r/\sqrt{t})\mathrm{d r}^2 + r^2A(r/\sqrt{t})\mathrm{d}\alpha ^2\;,
  \end{align}
  where $R,A:(0,\infty)\to (0,1]$ are bounded functions.
\end{proposition}
Of course, the precise form of the functions $R$ and $A$ is different
for $C(\pi/2)$ and $C(\pi)$.

\begin{proposition}\label{prop:convergence2}
  As $t$ goes to zero, the metric space $(C(\pi/2),d_t)$ tends to
  $(C(\pi/2),d)$ pointwise and in the pointed Gromov--Hausdorff
  topology.
\end{proposition}

In order to calculate the angle at the apex, for $r>0$ let us set
again $\rho_t(r)=d_t\big(o,(0,r)\big)$, where $o$ denotes the apex,
as well as
\begin{align*}
  l_t(r)=\int_0^{\pi/2}|\dot p^r_s|_t\dd s\;.
\end{align*}
where the curve $p^r:[0,\pi/2]\to
C(\pi/2)=\big([0,\infty)\times[0,\pi/2]\big)\big/\sim$ is given by
$p^r_s=(r,s)$.

\begin{theorem}\label{thm:singular2}
   For each $t>0$ we have
  \begin{align*}
    \lim\limits_{r\to0}\frac{l_t(r)}{\rho_t(r)}=0\;.
  \end{align*}
  In other words, the angle at the apex $o$ is zero. We have more precisely $R(r)\sim
  r^2/4$ and $A(r)\in O(r^3)$ as $r\to 0$.
  
  Moreover, $\R^+$ is both, the tangent space of $(C(\pi/2),d_t)$ at $o$, and the limit in the pointed Gromov--Hausdorff topology as $t$ goes to infinity.
\end{theorem}

\begin{proof}
  We will follow the same reasoning as in the proof of Theorem
  \ref{thm:singular}. Let us highlight the main steps. By scaling, we
  can again assume that $t=1$. Let us denote by $\mu^1_z=\frac14
  \big[\gamma^1_z +\gamma^1_{iz}+\gamma^1_{-z}+\gamma^1_{-iz}\big]$
  the lift of $\nu^1_{P(z)}$ and denote by $f_z$ its density with
  respect to the Lebesgue measure. Recalling the expressions
  \eqref{eq:rho1}, \eqref{eq:l1} (the latter with $\pi$ replaced by
  $\pi/2$) for $\rho$ and $l$, it is sufficient to calculate $R(r)$
  and $A(r)$ asymptotically as $r\to0$.  A rotation of $\pi/4$ permits
  us to see the measures $\mu_z^t$ as product measures. This allows us
  to calculate $R$ exactly in a similar way as in Theorem
  \ref{thm:singular}. Thus, we set $z_r=re^{i\pi/4}$ and recall that
  $R(r)=\norm{V_r}^2_{L^2(\mu^1_{z_r};\R^2)}$, where $V_r$ is the
  unique vector field in $T_{\mu^1_{z_r}}\cP(\R^2)$ solving the continuity
  equation
  \begin{align}\label{eq:ceR-small}
    \frac{\dd}{\dd h}\Big|_{h=0}f_{z_r+he^{i\pi/4}} +\diverg(f_{z_r}V_r) = 0\;.
  \end{align}
  Using the explicit expression
  \begin{align*}
    f_{z_r}(z)= \frac14 \Big[\eta(x_1-r/\sqrt{2}) + \eta(x_1+r/\sqrt{2})\Big]\Big[\eta(x_2-r/\sqrt{2}) + \eta(x_2+r/\sqrt{2})\Big]\;,
  \end{align*}
  we readily check that  
  \begin{align*}
    &\frac{\dd}{\dd h}\Big|_{h=0}f_{z_r+he^{i\pi/4}} =\\
  &\frac1{4\sqrt2} \Big[-\eta'(x_1-r/\sqrt{2}) + \eta'(x_1+r/\sqrt{2})\Big]\Big[\eta(x_2-r/\sqrt{2}) + \eta(x_2+r/\sqrt{2})\Big]\\
  +&\frac1{4\sqrt2} \Big[\eta(x_1-r/\sqrt{2}) + \eta(x_1+r/\sqrt{2})\Big]\Big[-\eta'(x_2-r/\sqrt{2}) + \eta'(x_2+r/\sqrt{2})\Big]\;,
  \end{align*}
  and that the solution to \eqref{eq:ceR-small} is given by
  \begin{align*}
    V_r(z)=\binom{\phi_r(x_1)}{\phi_r(x_2)}\;,\qquad \phi_r(x)=\frac1{\sqrt{2}}\frac{\eta(x+r/\sqrt{2})-\eta(x-r/\sqrt{2})}{\eta(x+r/\sqrt{2})+\eta(x-r/\sqrt{2})}\;.
  \end{align*}
  Thus, we find
  \begin{align*}
    R(r)&=\frac12\int_{\R^2} |V_r|^2f_{z_r} = \frac18 \int\frac{|\eta(x+r/\sqrt{2})-\eta(x-r/\sqrt{2})|^2}{\eta(x+r/\sqrt{2})+\eta(x-r/\sqrt{2})}\dd x\\
     &=\frac{r^2}{4} + o(r^2)\;.
  \end{align*}
  Let us turn to calculating $A$. Here, it is convenient to set
  $z_r=(r,0)$ and recall that
  $A(r)=\norm{V_r}^2_{L^2(\mu^1_{z_r};\R^2)}$, where $V_r$ is the
  unique solution in $T_{\mu^1_{z_r}}\cP(\R^2)$ to the continuity equation
  \begin{align}\label{eq:ceA-small}
    \frac{\dd}{\dd h}\Big|_{h=0}f_{z_r+(0,h)} +\diverg(f_{z_0}V_r) = 0\;.
  \end{align}
  We will again approximate the solution. First note that
  \begin{align*}
    \beta_r&= \frac{\dd}{\dd h}\Big|_{h=0}f_{z_r+(0,h)}(z)\\
&= \frac14\Big[\eta'(x_1)\Big(\eta(x_2-r)-\eta(x_2+r)\Big) - \eta'(x_2)\Big(\eta(x_1-r)-\eta(x_1+r)\Big)\Big]\\
&= \frac{r^3}{12}\Big[\eta'''(x_1)\eta'(x_2) - \eta'(x_1)\eta'''(x_2)\Big] + o(r^3) \\
&=\frac{r^3}{48}(x_1^3x_2-x_1x_2^3)\eta(x_1)\eta(x_2) + o(r^3) \;,
  \end{align*}
  where in the last equality we have used that $\eta'(u)=-(u/2)\eta(u)$
  and $\eta'''(u)=(-(u^3/4)+(5u/4))\eta(u)$. Now, set
  $\psi^r(y)=\frac{r^3}{96}[y_1^3y_2-y_1y_2^3]$ and calculate
  \begin{align*}
    \delta_r =& -\diverg(f_{z_r}\nabla\psi^r) = -\partial_1f_{z_r}\partial_1\psi^r - \partial_2f_{z_r}\partial_2\psi^r \\
            =& \frac{r^3}{96}(3x_1^2x_2-x_2^3)\frac14\Big[\eta'(x_1)\big(\eta(x_2-r) + \eta(x_2+r)\big)\\
            &\qquad\qquad\qquad\qquad\qquad+ \eta(x_2)\big(\eta'(x_1-r) + \eta'(x_1+r)\big)\Big]\\
&- \frac{r^3}{96}(3x_2^2x_1-x_1^3)\frac14\Big[\eta(x_1)\big(\eta'(x_2-r) + \eta'(x_2+r)\big)\\
            &\qquad\qquad\qquad\qquad\qquad+ \eta'(x_2)\big(\eta(x_1-r) + \eta(x_1+r)\big)\Big]\\
           =&\frac{r^3}{96}\Big[(3x_1^2x_2-x_2^3)\eta'(x_1)\eta(x_2) - (3x_2^2x_1-x_1^3)\eta(x_1)\eta'(x_2)\Big] +o(r^3)\\
           =&\frac{r^3}{48}(x_1^3x_2-x_1x_2^3)\eta(x_1)\eta(x_2) +o(r^3)\;.
 \end{align*}
 Hence, as $r\to0$ we obtain that
 \begin{align*}
   \frac1{r^6} \int \frac{|\beta^r -\delta^r|^2}{f_{z_r}} \rightarrow
   0\;.
 \end{align*}
 Since the measures $\mu^1_{z_r}$ satisfy the Poincar\'e inequality
 with constant independent of $r$, we deduce by Lemma
 \ref{lem:H-1vsL2} that
 \begin{align*}
   \frac{1}{r^3}\norm{\nabla\psi^r
   -V_{r}}_{L^2(\mu^1_{z_r};\R^2)}\rightarrow 0\;.
 \end{align*}
 This yields that
 $\sqrt{A(r)}=\norm{\nabla\psi^r}_{L^2(\mu^1_{z_r})}+o(r^3)=Cr^3
 +o(r^3)$ for a suitable constant $C$. Using finally \eqref{eq:rho1},
 \eqref{eq:l1} we find that $\rho_1(r)$ is of order $r^2$, while
 $l_1(r)$ is of order $r^3$. This yields the claim on the ratio.

In analogy with the end of the proof of Theorem \ref{thm:singular} concerning the transformation $\mathcal{T}$, and with the notation adapted from it we find $\bar{R}=1$ and $\bar{A}=o(1)$ when $\bar r$ goes to zero. One recognizes that the tangent cone is $\R^+$ and, using the space-time scaling similarly to Lemma \ref{lem:scaling} one sees that $\R^+$ is also the pointed Gromov--Hausdorff limit when $t$ goes to infinity.
\end{proof}

\section{Smoothing the Heisenberg group}
\label{sec:subriem}

\subsection{Heisenberg group}
Most of the considerations in this section can be generalized to the
higher-dimensional Heisenberg groups, but for simplicity we consider only the first Heisenberg group $\H$. This Lie group can be represented by
$\H=\C\times\R$ with the multiplicative structure
\begin{align*}
  (z,u)\cdot(z',u')=\left(z+z',u+u'-\frac12\mathrm{Im}(z\bar{z'})\right)\;,
\end{align*}
where $\mathrm{Im}$ is the imaginary part of a complex number. 

A basis for the Lie algebra is given by the left invariant vector
fields
\begin{align*}
  \vx=\partial_x-\frac{y}{2}\partial_u\;,\quad \vy=\partial_y+\frac{x}{2}\partial_u\;,\quad \vu=\partial_u\;,
\end{align*}
and the relation $\bra{\vx,\vy}=\vu$. 
We will also consider the right invariant vector
fields
\begin{align*}
  \hat\vx=\partial_x+\frac{y}{2}\partial_u\;,\quad \hat\vy=\partial_y-\frac{x}{2}\partial_u\;,\quad \hat\vu=\vu\;.
\end{align*}
The Haar measure associated with the group structure is up to a
constant multiple the 3-dimensional Lebesgue measure, denoted by
$\cL$, it is both left- and right-invariant.

\subsection{Riemannian and sub-Riemannian distances}

The Heisenberg group carries a sub-Riemannian structure given by the pseudo-norm
\begin{align*}
\norm{a\vx+b\vy+c\vu}^2_{\CC}=
\begin{cases}
 a^2+b^2&\text{if $c=0$,}\\
+\infty&\text{otherwise.}
\end{cases}
\end{align*}
The Carnot--Carath\'eodory distance $\dcc$ is obtained by minimizing
the sub-Riemannian length of curves connecting two points. More
precisely, given $p,q\in\H$ we have
\begin{align*}
  \dcc(p,q)=\inf\int_0^T\norm{\dot\gamma_s}_{\CC}\dd s\;,
\end{align*}
where the infimum is taken e.g.~over all absolutely continuous curves
$(\gamma_s)_{s\in[0,T]}$ with respect to the Euclidean distance such
that $\gamma_0=p,\gamma_T=q$. Note that the sub-Riemannian length of
$\gamma$ is only finite if $\gamma$ is horizontal, i.e.~ for a.e.~$s$
the tangent vector $\dot\gamma_s$ is contained in the horizontal
sub-bundle 
$$\mathrm{T}\H=\Vect(\vx,\vy)\;.$$
As a consequence of the so-called H\"ormander
condition
, namely that the horizontal vector fields generate the full tangent
space, the distance $\dcc$ is finite: any two points of $\H_n$ can be
connected by a horizontal curve of finite length and even a minimizing curve can be found.
Note that a curve is absolutely continuous with
respect to the Carnot--Carath\'eodory distance if and only if it is
absolutely continuous with respect to the Euclidean distance, its
tangent vector is horizontal at almost every point and its
sub-Riemannian length is finite.

The $3$-dimensional Lebesgue measure $\cL$ coincides with the
$4$-dimension Hausdorff measure of the metric space $(\H,\dcc)$.  It
has been shown in \cite{Jui09} that the metric measure space
$(\H,\dcc,\cL)$ does not satisfy the curvature-dimension condition
CD$(K,N)$ for any $K,N$.

However, the sub-Riemannian pseudo-norm is naturally approximated by a family
of Riemannian metrics indexed by $\eps>0$ and defined via
\begin{align*}
\norm{a\vx+b\vy+c\vu}^2_{\Rm(\eps)}=a^2+b^2+(c/\eps)^2\;.
\end{align*}
We denote the associated Riemannian distance by $d_\Rm(\eps)$. The
associated Riemannian volume coincides with $\cL$ up to a constant.
One can check that the best lower bound on the Ricci curvature of
$\norm{\cdot}_{\Rm(\eps)}$ is $-\frac12 \eps^{-2}$, see
e.g.~\cite{BBBC}. We have the following comparison of $d_{cc}$ and
$d_\Rm$, that stand for $d_{\Rm(1)}$.
\begin{proposition}[{\cite[Lemma1.1]{Ju_grad}}]\label{compare}
  We have 
  \begin{align*}
    d_\Rm\leq \dcc\leq d_\Rm+4\pi\;.
  \end{align*}
  Moreover, there are positive constants $c$ and $C$ such that for any
  point $p=(z,u)\in\H=\C\times\R$:
  \begin{align*}
    \max(|z|,c(|z|+|u|^{1/2}))\leq \dcc(0_\H,p)\leq C(|z|+|u|^{1/2})\;.
  \end{align*}
\end{proposition}

\subsection{Isometries}\label{trans_dil}
For every $p\in \H$, we denote by $\tau_p:\H\to \H$ and $\theta_p:\H\to \H$ the left and
right translations respectively, i.e.~
\begin{align*}
  \tau_p(q)=p\cdot q=\theta_q(p)\;.
\end{align*}
By definition a vector field $V$ is a left invariant if and only if
$D\tau_p(V)=V$ for every $p\in\H$. Hence, the left translation
$\tau_p$ is an isometry for both distances $\dcc$ and $d_\Rm$. This
is false for $\theta_q$ unless $q=0_\H$.

Other isometries are
\begin{itemize}
\item the rotations $\rho_\alpha:\H\ni(z,u)\mapsto(\mathrm{e}^{
      i\alpha}z,u)$ defined for $\alpha\in\R$,
\item the reflection $\xi: \H\ni (z,u)\mapsto (\bar{z},-u)$,
\end{itemize}
and up to the multiplicative constant $\lambda$,
\begin{itemize}
\item the dilations $\delta_\lambda:\H\ni (z,u)\mapsto (\lambda
  z,\lambda^2u)$ where $\lambda>0$.
\end{itemize}
One has $D\delta_\lambda(V)=\lambda V$ if and only if $V$ is
horizontal. In general one has 
\begin{align*}
  D\delta_\lambda(a\vx+b\vy+c\vu)=\lambda
  (a\vx+b\vy)+\lambda^2c\vu\;.
\end{align*}
Therefore $\delta_\lambda$ is an isometry between $(\H,\dcc)$ and
$(\H,\lambda^{-1} \dcc)$ as well as between $(\H,d_{\Rm(\eps)})$ and
$(\H,\lambda^{-1}d_{\Rm(\eps\lambda)})$. Hence, all the Riemannian
manifolds $(\H,d_{\Rm(\eps)})_{\eps>0}$ are isometric up to a
multiplicative constant, which justifies that we mainly consider
$(\H,d_{\Rm})$ corresponding to $\eps=1$.

\subsection{Wasserstein space over the Heisenberg group and absolutely
  continuous curves of measures}

Denote by $W_{\H}$ the $L^2$-Wasserstein distance build from the
Carnot--Carath\'eodory distance $\dcc$. We will recall here the
characterization of $2$-absolutely continuous curves in
$\big(\cP_2(\H,\dcc),W_\H\big)$ via solutions to the continuity
equation.

Denote by $\diverg \vv$ the divergence of a vector field
$\vv$ on $\R^3$ with respect to the Lebesgue measure. Note that the basis vector fields $\vx,\vy,\vu$ all
have divergence zero and moreover, we have
$\diverg(f\vx + g\vy + h\vv)=\vx f + \vy g + \vu h$ for every smooth
functions $f,g,h$. We denote by
\begin{align*}
  \nabla_\H f = (\vx f)\vx + (\vy f)\vy\;,
\end{align*}
the horizontal gradient of a function $f$. Then, for any smooth,
compactly supported function $f$ and vector field $\vv$ we have the
integration by parts formula
\begin{align*}
  \int_{\H} f \diverg \vv\dd\cL = - \int_\H \langle\nabla_\H f,\vv\rangle_{\CC} \dd\cL\;.
\end{align*}
Further let us denote by $L^2_{\CC}(\mu)$ the Hilbert space of Borel
vector fields $\vv$ equipped with the norm
\begin{align*}
\norm{\vv}_{L^2_{\CC}}^2=\int\norm{\vv}^2_{\CC}\dd \mu\;.
\end{align*}
Note that any $\vv\in L^2_{\CC}(\mu)$ must be horizontal
$\mu$-a.e. Now, we have the following characterization of absolutely
continuous curves.
\begin{proposition}[{\cite[Proposition
  3.1]{Ju_grad}}]\label{prop:heisenberg-ac-curves}
  A weakly continuous curve $(\mu_s)_{s\in [0,T]}$ in $\cP_2(\H)$ is
  $2$-absolutely continuous with respect to $W_\H$ if and only if there
  exists a Borel family of vector fields $\vv_s$ with
  $\int_0^T\norm{\vv_s}_{L^2_{\CC}(\mu_s)}^2\dd s<\infty$ such that the
  continuity equation
  \begin{align*}
    \partial_s\mu +\diverg(\mu_s\vv_s)=0
  \end{align*}
  holds in distribution sense. In this case we have $|\dot \mu_s| \leq \norm{\vv_s}_{L^2_{\CC}(\mu_s)}$ for a.e. $s$.
  Moreover, $\vv_s$ is uniquely determined for a.e.~$s$ if we require
  \begin{align*}
  \vv_s\in T_{\mu_s}\cP_2(\H) := \overline{ \{ \nabla_\H \psi \ | \
    \psi \in C^\infty_c(\R^3) \}}^{L^{2}_{\CC}(\mu_s)}
  \end{align*}
  and there holds $|\dot \mu_s|=\norm{\vv_s}_{L^2_{\CC}(\mu_s)}$.
\end{proposition}
 
Following verbatim the argument of Lemma \ref{lem:H-1vsL2} we obtain a
similar statement in the Heisenberg group.

\begin{lemma}\label{lem:Heisenberg-H-1vsL2}
  Let $\mu=\rho\cL\in \cP_2(\H)$ with strictly positive density
  $\rho$. Assume that $\mu$ satisfies the Poincar\'e type inequality
  \begin{align}\label{eq:Poinc-heis}
    \int |f|^2\dd\mu \leq C \int \norm{\nabla_\H f}_{cc}^2 \dd\mu\;,
  \end{align}
  for all $f\in C^\infty_c(\H)$ with $\int f\dd\mu =0$. Let $s\in
  L^1(\H,\cL)$ be such that $\int s\dd\cL=0$ and
  \begin{align*}
    \norm{s/\sqrt{\rho}}^2_{L^2} ~=~ \int \frac{s^2}{\rho}\dd \cL ~<~\infty\;.
  \end{align*}
  Then there exists a unique horizontal vector field
  $V\in T_{\mu}\cP_2(\H)$ such that the equation
  \begin{align*}
    s+\diverg(\mu V)=0
  \end{align*}
  holds in distribution sense. Moreover, we have
  \begin{align}\label{eq:H-1vsL2-heis}
    \norm{V}_{L^2_{\CC}(\mu)}^2 ~\leq~ C \norm{s/\sqrt{\rho}}^2_{L^2}\;.
  \end{align}
\end{lemma}

\subsection{Heat kernel}\label{heat_ker}

Another important consequence of the H\"ormander condition is the
hypoellipticity of the operators $\Delta_{\CC}=\vx^2+\vy^2$ and
$\Delta_{\CC}-\partial_t$, which in particular means that
distributional solutions $\rho:(0,\infty)\times \H\to \R$ of the heat
equation
\begin{align*}
\partial_t\rho=\Delta_{cc} \rho,
\end{align*}
are smooth. Note that the heat equation is left invariant. As shown by
Gaveau \cite{Ga}, the unique distributional solution $\mu_t=\rho_t\cL$
with initial condition $\mu_0\in\cP_2(\H)$ is given via convolution
with a fundamental solution $\h_t$:
\begin{align*}
  \rho_t(p)= \int \h_t(q^{-1}p)\dd\mu_0(q)\;,
\end{align*}
where $\h_t$ is given explicitly by
\begin{align*}
  \h_t(z,u)=\frac{2}{(4\pi t)^{2}}\int_{\R}\exp\left(\frac{\lambda}{t}\left(iu-\frac{|z|^2}{4}\coth \lambda\right)\right)\frac{\lambda}{\sinh \lambda}\,\dd\lambda.
\end{align*}
In fact $\h_t$ is the density of $X_t=(B_{2t},L_{2t})$
where the process $(B_t)_{t\geq0}$ is a planar Brownian motion
$B=B^1+iB^2$ and
$L_t=\frac12\int_0^t(B_s^1\dd B_s^2-B_s^2\dd B_s^1)$ is the L\'evy
area. Hence $\h_t$ is a strictly positive probability density with
respect to $\cL$ for all $t$. Moreover, $\h_t\cL\in\cP_2(\H)$.

We will need the following estimates \cite[(14) and proof of Thm. 3.1]{BBBC}.
\begin{align}\label{finite_entropy}
\int\big((\vx \log\h_t)^2 + (\vy \log\h_t)^2\big)\h_t\dd\cL=\frac{2}{t}\;,\quad \int(\vu \log\h_t)^2\h_t\dd \cL<\infty\;.
\end{align}
The same estimates hold for the right invariant vector fields
$\hat\vx,\hat\vy,\hat\vu$.
Note also the scaling relation
\begin{align}\label{eq:scaling-heis}
\h_t(z,u)=\frac{1}{t^{2}}\h_1(z/\sqrt{t},u/t)\;.
\end{align}
Given $t\geq0$ and $q\in\H$ we define the measure
$\nu^t_q\in\cP_2(\H)$ via
\begin{align*}
\nu^t_q =
\begin{cases}
\delta_q&\text{if }t=0\;,\\
(\tau_q)_\#(\h_t\cL)=\h_t(q^{-1}p)\cL(\mathrm{d} p)&\text{otherwise,}
\end{cases}
\end{align*}
and call it the heat kernel measure centered at $q$.

\begin{lemma}\label{lem_injective}
  The map $\iota_t:(\H,d_{\CC})\ni q \mapsto \nu^t_q\in (\cP_2(\H),W_\H)$
  is injective and Lipschitz. Moreover, $W_\H(\nu^t_p,\nu^t_q)$ tends to
  infinity as $d_{\CC}(p,q)$ goes to infinity.
\end{lemma}
Before we go to the proof, let us stress that the isometries of
$(\H,\dcc)$ introduced in paragraph \ref{trans_dil} give rise to
isometries of $(\cP_2(\H),W_\H)$ via pushforward. In particular,
translations of measures $(\tau_p)_\#$ are isometries.

\begin{proof}
  We first check injectivity. Form the probabilistic interpretation
  of $\h_t$ one sees that the projection $P^\C:\H\to\C$ transports
  $\nu^t_{(z,u)}$ to the $2$-dimensional Gaussian measure centered at
  $z$ with covariance $2t\cdot\mathrm{Id}$. Therefore, we have
  $\nu^t_{(z,u)}\neq\nu^t_{(z',u')}$ provided $z\neq z'$. Further,
  note that $\nu^t_{(z,u)}=(\tau_{(0,u'-u)})_\#\nu^t_{(z,u')}$. Hence,
  the two measures are distinct provided $u\neq u'$ since
  $(\tau_{(0,u'-u)})_\#$ is an isometry. This proves injectivity of
  $\iota_t$.
  
  Lipschitz continuity of $\iota_t$ follows from Kuwada's duality between
  $L^q$-Wasserstein contraction estimates and $L^p$-gradient estimates
  on the heat kernel of the Heisenberg group, established e.g.~in
  \cite{BBBC}. See \cite[Prop. 4.1]{Ku} and the discussion thereafter.
  The distance $W_\H(\delta_p,\nu^t_p)=W_\H(\delta_{0_\H},\nu)$ is
  independent from $p$. Hence
  \begin{align*}
    W_\H(\nu_p^t,\nu_q^t)\geq W_\H(\delta_p,\delta_q)-2W_\H(\delta_{0_\H},\nu)
  \end{align*}
  tends to infinity as $\dcc(p,q)=W_\H(\delta_p,\delta_q)$ tends to
  infinity.
\end{proof}

Finally, we note that for any $t>0$ the measures $\nu_q^t$ satisfy the
Poincar\'e type inequality \eqref{eq:Poinc-heis} with a constant
$C=c\cdot t$ for some $c>0$, see \cite[Thm.~1.7]{DM05}.

\subsection{Smoothing effect of the transformation}

In view of Lemma \ref{lem_injective} the map
$\iota_t:\H\ni q \mapsto \nu^t_q\in \cP_2(\H)$ satisfies Assumption
\ref{ass:iota} and the hypotheses of Proposition \ref{prop:convergence} are also satisfied. Thus as outlined in Section \ref{sec:construction} we
can introduce the new distance
\begin{align*}
  \tilde d_t(p,q)= W_\H(\nu^t_p,\nu^t_q)\;,
\end{align*}
as well as the associated length distance
\begin{align*}
 d_t(q,p)~=~\inf\int_0^T|\dot p_s|_t\dd s = \inf \int_0^T |\dot \nu^t_{p_s}|\dd s\;,
\end{align*}
where the infimum is taken over absolutely continuous (or equivalently
Lipschitz) curves $(p_s)_{s\in[0,T]}$ in $(\H,\tilde d_t)$ such that
$p_0=p,p_T=q$ and $|\dot p_s|_t$ and $|\dot \nu^t_{p_s}|$ denote the
metric derivatives with respect to $\tilde d_t$ and $W_\H$
respectively.

We have the following scaling relation.
\begin{proposition}\label{scaling}
For $t>0$ and every $p,q\in\H$ we have
\begin{align}\label{eq:scaling-heis-dt}
 d_{t}(\delta_{\sqrt{t}}p,\delta_{\sqrt{t}}q) = \sqrt{t}\cdot d_1(p,q)\;.
\end{align}
In other words, the dilation $\delta_{\sqrt{t}}$ is an isometry from
$(\H,\sqrt{t}d_1)$ to $(\H,d_t)$.
\end{proposition}
\begin{proof}
  The measure dilation $(\delta_\lambda)_\#$ dilates the Wasserstein
  distance $W_\H$ by the factor $\lambda$. As a consequence of the
  scaling relation \eqref{eq:scaling-heis} we find that
  $(\delta_\lambda)_\#\nu^t_q=\nu^{\lambda^2t}_{\delta_\lambda(q)}$. Thus,
  we obtain \eqref{eq:scaling-heis-dt} with $d$ replaced by
  $\tilde d$, which then easily passes to the induced length distance.
\end{proof}

We can now state our main theorem about the smoothing effect of the
transformation of the distance.

\begin{theorem}\label{thm:Heisenberg-main}
  The distance $d_t$ is induced by a left-invariant Riemannian metric
  tensor $g_t$. More precisely, we have $d_t=K\cdot
  d_{\Rm(\kappa\sqrt{t})}$, where the constant $K,\kappa$ satisfy
  $K\geq2$ and $K/\kappa<\sqrt2$.
\end{theorem}

The numerical estimates on $\kappa$ and $K$ will be given in Remarks
\ref{rem:elevator} and \ref{rem:constant_K}. The first remark explains
the reason why the convolution procedure allows to recover the
forbidden non-horizontal direction. The second remark relates $K$ to the optimal constant in Wasserstein
contraction estimates for the heat flow. Together with the
convergence results in Proposition \ref{prop:approximation-heisenberg},
this result proves Theorem \ref{thm:main-H} claimed in the
introduction.

\begin{proof}[Proof of Theorem \ref{thm:Heisenberg-main}]
  By Proposition \ref{scaling} it suffices to consider $t=1$ and we
  will do so for the moment. We suppress the index $t=1$ in the
  notation, setting $d=d_1, \tilde d=\tilde d_1$ and $\nu_q=\nu^1_q$.

\medskip
 \emph{Definition of $g$:}~~
  For $q\in\H$ and $a,b,c \in\R$ we define:
  \begin{align}
    g_q\big((a\vx +b\vy+c\vu)(q)\big) = \norm{\vv^{a,b,c}_q}^2_{L^2_{\CC}(\nu_q)}\;,
  \end{align}
  where $\vv^{a,b,c}_q$ is the unique vector field in
  $T_{\nu_q}\cP_2(\H)$ solving the continuity equation
  \begin{align}\label{eq:ce-def-g}
    \partial_s\big|_{s=0}\nu_{q_s} + \diverg(\nu_q\vv^{a,b,c}_q)=0
  \end{align}
  with a curve $(q_s)_s$ such that $q_0=q$ and
  $\dot q_s=(a\vx +b\vy+c\vu)(q_s)$. Existence and uniqueness of
  $\vv^{a,b,c}_q$ are ensured by Lemma \ref{lem:Heisenberg-H-1vsL2}. We
  will show below that $g$ is indeed a metric tensor after polarizing it
  to a bilinear form. First we have to check that the assumptions of
  Lemma \ref{lem:Heisenberg-H-1vsL2} are fulfilled.

  Note that the continuity equation can be rewritten as
  \begin{align}\label{cont_hei}
    \partial_s\big|_{s=0}\rho_s + \diverg(\rho_0\vv^{a,b,c}_q)=0\;,
  \end{align}
  where $\rho_s(p)=\h_1(q_s^{-1}p)$ is the density of
  $\nu_{q_s}$. The derivation of $s\mapsto q_s^{-1}.q_s$ yields
 \begin{align*}
    \frac{\mathrm{d}}{\mathrm{d}s}q_s^{-1} = -q_s^{-1}\dot q_sq_s^{-1} = -(a\hat\vx + b\hat\vy +c\hat\vu)(q_s^{-1})
  \end{align*}
where we use the equalities between left- and right-invariant vector fields at $0_\H=q_s^{-1}.q_s$.
We obtain
  \begin{align*}
    \partial_s\big|_{s=0}\rho_s(p) = - \big((a\hat\vx+b\hat\vy+c\hat\vu)\h_1\big)(q^{-1}p)\;.
  \end{align*}
  Thus, by the left invariance of $\cL$ and \eqref{finite_entropy}, we
  have that
  \begin{align*}
    \int\frac{|\partial_s|_{s=0}\rho_s|^2}{\rho_0}\dd\cL = \int
    \frac{|(a\hat\vx+b\hat\vy+c\hat\vu)\h_1|^2}{\h_1}\dd\cL<\infty\;.
  \end{align*}
  Hence, Lemma \ref{lem:Heisenberg-H-1vsL2} is indeed applicable.

\medskip
\emph{The metric $g$ is Riemannian:}~~
By linearity of \eqref{eq:ce-def-g} and uniqueness of the solution, $\vv^{a,b,c}_q$
depends linearly on $a,b,c$. Thus $g_q(\cdot)$ is quadratic and indeed
gives rise to a metric tensor after polarization. Note moreover, that
$\vv^{a,b,c}_q=D\tau_q(\vv^{a,b,c}_{0_\H})$. This implies that $g$ is
left invariant, i.e.~
\begin{align*}
  g_q\big((a\vx+b\vy+c\vu)(q)\big) = g_{0_\H}\big((a\vx+b\vy+c\vu)(0_\H)\big)\;.
\end{align*}
In particular, $g_q$ depends smoothly on $q$.

\medskip

\emph{Characterization of the Riemannian metrics obtained by
  convolution:}~~ It is readily checked that $g$ is also invariant
under rotations $\rho_\alpha$. Left invariant and rotation invariant
Riemannian metrics $g$ on $\H$ form a two parameter family indexed by
$K,\kappa>0$ defined by $K=g(\vx)^{1/2}=g(\vy)^{1/2}$ and
$K/\kappa=g(\vu)^{1/2}$. Thus, we must have that
$g=K^2\norm{\cdot}^2_{\Rm(\kappa)}$.

\medskip

\emph{The distance $d$ coincides with the Riemannian distance:}~~ 
Let us denote by
$d_g$ the Riemannian distance induced by $g$ and recall that
\begin{align*}
  d_g(p,q) = \inf \int_0^T \sqrt{g_{q_s}(\dot q_s)}\dd s\;,
\end{align*}
where the infimum is taken e.g.~ over all curves $(q_s)_{s\in[0,T]}$
connecting $p$ to $q$ that are Lipschitz with respect to
Euclidean distance. To see that $d$ coincides with $d_g$, it is
sufficient to check that a curve $(q_s)_{s}$ is $\tilde d$-Lipschitz if and only if it is locally Lipschitz in
Euclidean sense and for any such curve we have
\begin{align}\label{eq:metdev}
  |\dot q_s|^2=g_{q_s}(\dot q_s)\;,
\end{align}
where the left hand side is the metric derivative with respect to
$\tilde d$.

So let $(q_s)_{s\in[0,T]}$ be a Euclidean Lipschitz curve such that
$\dot q_s=(a_s\vx+b_s\vy+c_s\vu)(q_s)$. Following the reasoning in the
first part of the proof, we see that the continuity equation
\begin{align*}
\partial_s\nu_{q_s}+\diverg(\nu_{q_s}\vv^{a_s,b_s,c_s}_{q_s})=0\end{align*}
holds with $\norm{\vv^{a_s,b_s,c_s}_{q_s}}^2_{L^2_{\CC}(\nu_{q_s})}=K^2(a_s^2+b_s^2
+ c_s^2/\kappa^2)$. Thus, by the characterization of absolutely continuous curves
in $(\cP_2(\H),W_\H)$, Proposition \ref{prop:heisenberg-ac-curves},
and the definition of $\tilde d$, the curve $(q_s)_s$ is locally
$\tilde d_t$-absolutely continuous with metric derivative
$K\sqrt{a_s^2+b_s^2 + c_s^2/\kappa^2}$, and also Lipschitz. Moreover,
\eqref{eq:metdev} holds by definition of $g$.  Conversely, to see that
any $\tilde d$-Lipschitz curve is also Euclidean Lipschitz,
it suffices to note that the previous argument shows in particular,
that $\tilde d\leq d_1\leq K\cdot d_{\Rm(\kappa)}$ and that $d_{\Rm(\kappa)}$
is locally equivalent to the Euclidean distance.
\end{proof}

\begin{remark}[Estimate on $\kappa$] \label{rem:elevator}
  The crucial feature of the regularized distance $d_t$ as opposed to
  $\dcc$ is that also non-horizontal curves can have finite length.
  This is due to the effect that even when the length of a curve
  $(q_s)_s$ with respect to $\dcc$ is infinite, the length of
  $(\nu^t_{q_s})_s$ with respect to the Wasserstein distance build from
  $\dcc$ may be finite. Let us make this more explicit for the special
  curve $q_s=(0,0,s)$. This curve is not horizontal and has infinite
  length, actually $\dcc(q_s,q_r)=c\cdot \sqrt{|s-r|}$, where
  $c=\dcc(0_\H,q_1)$, which follows from the behavior of $\dcc$ under
  translations and dilations. However, the curve
  $\nu^t_{q_s}=\rho^t_s\cL$ with $\rho_s(p)=\h_t(q_s^{-1}p)$ satisfies
  the continuity equation
  \begin{align*}
    \partial_s\rho_s = -\vu \rho_s = -\left[\vx,\vy\right]\rho_s=-\vx(\vy\rho_s)+\vy(\vx\rho_s) = -\diverg(\rho_s\vv_s)
  \end{align*}
  with the horizontal, but not of gradient type vector field
  $\vv_s=(\vy\log\rho_s)\vx - (\vx\log\rho_s)\vy
  $. Hence, we have
  \begin{align*}
    |\dot \nu_{q_s}|<\norm{\vv_s}_{L^2_{\CC}}= \sqrt{\int\frac{(\vx\h_t)^2+(\vy\h_t)^2}{\h_t}\dd\cL}=\sqrt{\frac{2}{t}} 
  \end{align*}
   by \eqref{finite_entropy}. Therefore the curve $(\nu_{q_s})$ has indeed finite
  $W_\H$-length and $K/\kappa=g_1(\vu)\leq \sqrt{2}$.
\end{remark}

\begin{remark}[Estimate on $K$] \label{rem:constant_K}
  It has been proved by Kuwada \cite{Ku} that the ratio $\tilde{d}_t/d_{cc}$
  is related to a gradient estimate established by Driver and Melcher
  \cite{DM05}. In fact  the constant $C_2$ in this estimate can be dually defined by
  \begin{align*}
    C_2=\sup_{p\neq q}d_{cc}(p,q)^{-1}W(\nu^t_p,\nu^t_q),
  \end{align*}
so that it is in particular independent from $t$.  From \cite{DM05} it is known $C_2\geq 2$ and a conjecture is $C_2=2$, see \cite[Remark
  3.2]{BBBC}. Let us show $K=C_2$, which gives a new understanding of this constant. We have
$$\tilde{d}_t/d_{cc}=(d_t/d_{cc})\times(\tilde{d}_t/d_t)$$
with $d_t/d_{cc}\leq K$ and $\tilde{d}_t/d_t\leq 1$. But for $q_s=(s,0,0)$ we see that that the quotient of the distances between $0$ and $q_s$ tend to $K$ and $1$ respectively as $s$ goes to zero. Therefore $g_1(\vx)=K=C_2\geq 2$.
\end{remark}

We conclude this section with an observation on the limiting behavior
of the convoluted distances as $t\to0$.

\begin{proposition}\label{prop:approximation-heisenberg}
  As $t\to 0$, for all $p,q\in\H$ we have:
  \begin{align*}
  \tilde d_t(p,q) &\rightarrow  \dcc(p,q)\;,\\
   d_t(p,q)       &\rightarrow  K\cdot \dcc(p,q)\;.
  \end{align*}
  Moreover, the metric spaces $(\H,d_t)$ converge to
  $(\H,d_{cc})$ in the pointed Gromov--Hausdorff sense.
\end{proposition}

\begin{proof}
  The pointwise convergence of $\tilde d_t$ to $\dcc$ follows from
  Proposition \ref{prop:convergence}. The pointwise convergence of
  $d_t$ follows immediately from the explicit formula for
  $d_t=Kd_{\Rm(\kappa\sqrt{t})}$ in Theorem
  \ref{thm:Heisenberg-main}. As the usual approximation of the
  subRiemannian Heisenberg group holds in the pointed
  Gromov--Hausdorff sense, the space $(\H,d_{t})$ tends to $(\H,
  Kd_{cc})$. But as explained in paragraph \ref{trans_dil},
  $(\H,K\cdot\dcc)$ is isometric to $(\H,\dcc)$ via the dilation
  $\delta_{K}$. The last statement follows.
\end{proof}

\bibliographystyle{abbrv} 
\bibliography{ricci}

\end{document}